\documentclass[11pt]{amsart}
\usepackage[T1]{fontenc}
\usepackage[utf8]{inputenc}
\usepackage{lmodern}
\usepackage{microtype}
\usepackage{xcolor}
\usepackage{amsmath,amssymb,amsthm,mathtools}
\usepackage{enumitem}
\usepackage[margin=1in]{geometry}
\usepackage{hyperref}
\usepackage{mathrsfs}
\hypersetup{colorlinks=true,linkcolor=blue,citecolor=blue,urlcolor=blue,
	pdftitle={Central Haagerup Tensor Products and Completely Bounded Maps under Strong Morita Equivalence},
	pdfauthor={Ilja Gogic}}

\theoremstyle{plain}
\newtheorem{theorem}{Theorem}[section]
\newtheorem{lemma}[theorem]{Lemma}
\newtheorem{proposition}[theorem]{Proposition}
\newtheorem{corollary}[theorem]{Corollary}
\theoremstyle{remark}
\newtheorem{remark}[theorem]{Remark}

\numberwithin{equation}{section}

\newcommand{\N}{\mathbb{N}}
\newcommand{\C}{\mathbb{C}}
\newcommand{\Link}{\mathscr{L}}
\newcommand{\cb}{\mathrm{cb}}
\newcommand{\M}{\operatorname{M}}
\newcommand{\cspan}{\overline{\operatorname{span}}}
\DeclareMathOperator{\CB}{CB}
\DeclareMathOperator{\Id}{Id}
\DeclareMathOperator{\Prim}{Prim}
\DeclareMathOperator{\Glimm}{Glimm}
\DeclareMathOperator{\len}{len}

\title[Central Haagerup Tensor Products under Morita Equivalence]{Central Haagerup Tensor Products and Completely Bounded Maps under Strong Morita Equivalence}

\author{Ilja Gogi\'{c}}

\address{I.~Gogi\'c, University of Zagreb Faculty of Science, Department of Mathematics, Bijeni\v{c}ka 30, 10000 Zagreb, Croatia}

\email{ilja@math.hr, ilja.gogic@math.pmf.unizg.hr}

\subjclass[2020]{Primary 46L07; Secondary 46L05, 46L06, 46L08, 47B47}

\keywords{central Haagerup tensor product, completely bounded map, elementary operator, strong Morita equivalence, imprimitivity bimodule, Glimm ideal, primal ideal}

\date{September 29, 2026.}

\begin{document}
	
	\begin{abstract}
		Let $A$ and $B$ be strongly Morita equivalent $C^*$-algebras, implemented by a nonzero imprimitivity bimodule $X={}_AX_B$ with its canonical operator-space structure. Denote by $Z_X$ their common multiplier centre, canonically identified through $X$. The coefficient action gives a complete contraction $\Theta_X:A\otimes_{Z_X,h}B\to\CB(X)$, where $A\otimes_{Z_X,h}B$ is the central Haagerup tensor product obtained by balancing over $Z_X$ and $\CB(X)$ is the space of completely bounded maps on $X$. We show that $\Theta_X$ is injective exactly when every Glimm ideal of $A$ is $2$-primal, and isometric, equivalently completely isometric, exactly when every such ideal is primal. For each positive integer $\ell$, norm preservation for sums of at most $\ell$ elementary tensors is equivalent to $(\ell^2+1)$-primality of every Glimm ideal. These conditions may equivalently be imposed on $B$. This extends the corresponding theorems of Somerset and Archbold--Somerset--Timoney to imprimitivity bimodules. We also show that, after matrix stabilization, the optimal constants bounding the central Haagerup norm by the completely bounded norm coincide for $X$, $A$ and $B$, at every finite tensor length and on the completed products, and agree with the corresponding constants after compact stabilization.
	\end{abstract}
	
	\maketitle
	
	\section{Introduction}\label{sec:introduction}
	
	A recurring problem in the theory of elementary operators on $C^*$-algebras is to determine when the Haagerup norm of a coefficient tensor equals the completely bounded norm of the operator it induces. The answer is closely related to the ideal structure of the underlying algebra. 
	
	Throughout, all ideals of $C^*$-algebras are understood to be closed and two-sided. We write $\N$ for the set of positive integers and $\odot$ for the algebraic tensor product. For a $C^*$-algebra $C$, let $\M(C)$ denote its multiplier algebra. Each tensor
	\[
	u=\sum_{j=1}^r a_j\otimes b_j\in\M(C)\odot\M(C),
	\qquad r\in\N,
	\]
	determines an \emph{elementary operator}
	\[
	T_u:C\to C,\qquad
	T_u(x):=\sum_{j=1}^r a_jxb_j.
	\]
	Its completely bounded norm is dominated by the \emph{Haagerup norm}
	\begin{equation}\label{eq:haagerup-norm}
		\|u\|_h:=
		\inf\left\{
		\left\|\sum_{j=1}^s c_jc_j^*\right\|^{1/2}
		\left\|\sum_{j=1}^s d_j^*d_j\right\|^{1/2}:
		u=\sum_{j=1}^s c_j\otimes d_j
		\right\},
	\end{equation}
	where the infimum ranges over all $s\in\N$ and $c_j,d_j\in\M(C)$, $1\le j\le s$. The completion of $\M(C)\odot\M(C)$ in this norm is the \emph{Haagerup tensor product} $\M(C)\otimes_h\M(C)$. With the canonical operator-space structures, the map $u\mapsto T_u$ extends to a complete contraction
	\begin{equation}\label{eq:canonical}
		\M(C)\otimes_h\M(C)\to\CB(C),
	\end{equation}
	where, for an operator space $E$, $\CB(E)$ denotes the space of completely
	bounded linear maps on $E$.
	
	Haagerup proved that \eqref{eq:canonical} is isometric for $C=\mathcal K(H)$ and hence also for $C=\mathcal B(H)$, where $\mathcal K(H)$ and $\mathcal B(H)$ denote the algebras of compact and bounded operators on a Hilbert space $H$, respectively \cite{Haagerup1980}; see \cite[Theorem~5.4.7 and Corollary~5.4.9]{AraMathieu2003}. Smith independently proved the same result for $\mathcal K(H)$ \cite[Theorem~4.3]{Smith1991}. Chatterjee and Sinclair subsequently proved that \eqref{eq:canonical} is isometric for every von Neumann factor with separable predual \cite[Theorem~3]{ChatterjeeSinclair1992}. The characterization was completed by Mathieu, who proved that, for every nonzero $C^*$-algebra $C$, primeness of $C$ is equivalent to injectivity of the algebraic coefficient map $\M(C)\odot\M(C)\to\CB(C)$, $u\mapsto T_u$, and also to isometry of its continuous extension \eqref{eq:canonical}; see \cite[Corollary~4.4]{Mathieu1989} and \cite[Proposition~5.4.11]{AraMathieu2003}, respectively.
	
	For a nonprime $C^*$-algebra, the multiplier centre need not be trivial. For a $C^*$-algebra $C$, we write $Z_C:=Z(\M(C))$ for this centre. If $a,b\in C$ and $z\in Z_C$, then the tensors $az\otimes b$ and $a\otimes zb$ induce the same elementary operator. This leads naturally to balancing over $Z_C$. Following \cite[Definition~3.1]{ArchboldSomersetTimoney2009}, define the \emph{central Haagerup tensor product} $C\otimes_{Z_C,h}C$ as the quotient of $C\otimes_h C$ by the closed linear span of all elements of the form $az\otimes b-a\otimes zb$, where $a,b\in C$ and $z\in Z_C$. Denoting by $a\otimes_{Z_C}b$ the image of $a\otimes b$ in this quotient, the coefficient action induces a complete contraction
	\[
	\theta_C:C\otimes_{Z_C,h}C\to\CB(C),\qquad
	\theta_C(a\otimes_{Z_C}b)(x):=axb,
	\]
	where $a,b,x\in C$. Here the tensor factors are $C$, rather than $\M(C)$ as in \eqref{eq:canonical}. This distinction is important in the nonunital case, as explained in Remark~\ref{rem:scope}.
	
	Chatterjee and Smith introduced the central Haagerup tensor product for unital $C^*$-algebras and proved that $\theta_C$ is isometric for von Neumann algebras and for unital $C^*$-algebras $C$ with Hausdorff primitive ideal space $\Prim(C)$ \cite[Theorems~2.4 and~3.1]{ChatterjeeSmith1993}. Ara and Mathieu obtained the analogous isometry result, with multiplier coefficients, for boundedly centrally closed $C^*$-algebras \cite[Theorem~3.7]{AraMathieu1994}, while Magajna gave another proof of the von Neumann algebra case using an injectivity theorem for module Haagerup tensor products \cite[Theorem~2.2 and Corollary~2.4]{Magajna1995}.
	
	The exact injectivity and isometry criteria are formulated in terms of Glimm and $n$-primal ideals, recalled in Section~\ref{subsec:cstar-facts}. For unital $C^*$-algebras, Somerset characterized injectivity of $\theta_C$ and proved that $\theta_C$ is isometric whenever every Glimm ideal is primal \cite[Corollary~6(ii) and Theorem~4]{Somerset1998}. Archbold, Somerset and Timoney subsequently proved the converse and established the exact bounded-length criterion \cite[Theorems~8 and~17]{ArchboldSomersetTimoney2005}. Their later work extended the injectivity, isometry and bounded-length criteria to arbitrary $C^*$-algebras \cite[Theorems~3.8 and~3.9]{ArchboldSomersetTimoney2009}. Somerset asked whether, for a unital $C^*$-algebra $C$, $\theta_C$ is completely isometric whenever every Glimm ideal of $C$ is primal \cite[p.~119]{Somerset1998}. By \cite[Remark~4.2]{ArchboldSomersetTimoney2009}, this hypothesis also implies that $\theta_{M_n(C)}$ is isometric for every $n\in\N$, where $M_n(C)$ denotes the $C^*$-algebra of $n\times n$ matrices over $C$. Using these matrix-level isometries and a column--row argument, Proposition~\ref{prop:complete-isometry} answers Somerset's question affirmatively and proves that, for every nonzero $C^*$-algebra $C$, isometry of $\theta_C$ implies complete isometry.
	
	The aim of this paper is to extend these single-algebra results to strong Morita equivalence. Let $A$ and $B$ be strongly Morita equivalent $C^*$-algebras in the sense of Rieffel \cite{Rieffel1974}, and let $X={}_AX_B$ be a nonzero imprimitivity bimodule implementing the equivalence, equipped with its canonical operator-space structure. The bimodule $X$ induces a canonical identification of the multiplier centres of $A$ and $B$, compatible with the Rieffel correspondence; see Section~\ref{subsec:ideals}. We therefore write
	\[
	Z_X:=Z(\M(A))=Z(\M(B))
	\]
	under this identification. As in the single-algebra case, we define the \emph{central Haagerup tensor product} of $A$ and $B$ by balancing over the common multiplier centre $Z_X$:
	\[
	A\otimes_{Z_X,h}B:=
	(A\otimes_hB)\Big/\cspan\{az\otimes b-a\otimes zb:
	a\in A,\ b\in B,\ z\in Z_X\},
	\]
	where $\cspan$ denotes closed linear span. For $a\in A$ and $b\in B$, write $a\otimes_{Z_X}b$ for the image of $a\otimes b$ in this quotient. Let $A\odot_{Z_X}B$ denote the image of $A\odot B$ in $A\otimes_{Z_X,h}B$, and write $\|\cdot\|_{Z_X,h}$ for the quotient
	norm. The left and right module actions induce the completely contractive
	\emph{coefficient map} associated with $X$:
	\begin{equation}\label{eq:theta-X}
		\Theta_X:A\otimes_{Z_X,h}B\to\CB(X),\qquad \Theta_X(a\otimes_{Z_X}b)(x):=axb,
	\end{equation}
	where $a\in A$, $b\in B$ and $x\in X$. For $0\ne u\in A\odot_{Z_X}B$, define its \emph{tensor length} by
	\[
	\len(u):=\min\left\{r\in\N:
	u=\sum_{j=1}^r a_j\otimes_{Z_X}b_j
	\text{ for some }a_j\in A,\ b_j\in B,\ 1\le j\le r\right\},
	\]
	and set $\len(0):=0$. The same convention applies to ordinary algebraic tensor products. When $X=A=B$, with the usual multiplication actions, $\Theta_X=\theta_A$.
	
	We can now state the main result of the paper.
	\begin{theorem}\label{thm:main}
		Let $A$ and $B$ be strongly Morita equivalent $C^*$-algebras, and let $X={}_AX_B$ be a nonzero imprimitivity bimodule implementing the equivalence. Then:
		\begin{enumerate}[label=\textup{(\roman*)}]
			\item $\Theta_X$ is injective if and only if every Glimm ideal of $A$ is $2$-primal.
			\item For each $\ell\in\N$, $\Theta_X$ preserves the norm of every $u\in A\odot_{Z_X}B$ with $\len(u)\le\ell$ if and only if every Glimm ideal of $A$ is $(\ell^2+1)$-primal.
			\item $\Theta_X$ is isometric if and only if it is completely isometric, and these conditions hold if and only if every Glimm ideal of $A$ is primal.
		\end{enumerate}
		Each ideal condition may equivalently be imposed on $B$. Consequently, injectivity, norm preservation at each fixed tensor length, isometry and complete isometry of $\Theta_X$ depend only on the strong Morita equivalence class, not on the particular implementing bimodule.
	\end{theorem}
	
	The proof of Theorem~\ref{thm:main} is given in Section~\ref{sec:main} and exploits the canonical realization of $X$ as the upper-right corner of its linking algebra $D$ from \eqref{eq:linking-algebra}. For the sufficient implications, Morita invariance of the relevant Glimm-ideal conditions and the full-corner norm identities transfer the problem from $\Theta_X$ to the single-algebra coefficient map $\theta_D$, where the known criteria apply (Lemmas~\ref{lem:morita-ideals} and~\ref{lem:corner-calculus}; Theorem~\ref{thm:diagonal}; Proposition~\ref{prop:complete-isometry}). For the necessary implications, we argue contrapositively. A tensor witnessing failure of injectivity or norm preservation for $\theta_D$ need not have coefficients in $A$ and $B$, so it cannot in general be used directly for $\Theta_X$. When the relevant Glimm-ideal condition fails, orthogonal lifting and the Rieffel correspondence produce matched finite families of positive coefficients in $A$ and $B$, yielding a completely isometric finite-dimensional tensor embedding into $A\otimes_{Z_X,h}B$ such that every primitive quotient annihilates at least one matching pair (Lemma~\ref{lem:orthogonal-lifting}; Proposition~\ref{prop:transport-obstruction}). The two-coordinate case yields a nonzero tensor in the kernel of $\Theta_X$, while \cite[Lemma~16]{ArchboldSomersetTimoney2005} provides bounded-length tensors witnessing failure of norm preservation. The necessity in the isometry criterion then follows from the finite-length result for all tensor lengths.
	
	We also study the \emph{inverse constants} of $\Theta_X$, namely the optimal constants in the reverse estimate $\|u\|_{Z_X,h}\le\lambda\|\Theta_X(u)\|_{\cb}$. For each $r\in\N$, the length-$r$ constant is the least $\lambda\in[1,\infty)$ for which this holds whenever $u\in A\odot_{Z_X}B$ and $\len(u)\le r$, with value $+\infty$ if no such finite $\lambda$ exists. The completed constant is defined analogously on $A\otimes_{Z_X,h}B$, and stabilization means taking the supremum over the matrix bimodules $M_n(X)$, $n\in\N$ (see \eqref{eq:inverse-constants}). We prove that, for every finite tensor length and in the completed case, the stabilized constants for $X$, $A$ and $B$ coincide and agree with those for the compact stabilization $X\otimes\mathcal K_0$, the external tensor product imprimitivity bimodule with coefficient algebras $A\otimes\mathcal K_0$ and $B\otimes\mathcal K_0$, where $\mathcal K_0$ is the $C^*$-algebra of compact operators on a separable infinite-dimensional Hilbert space (Theorem~\ref{thm:full-corner-constants}; Corollary~\ref{cor:constants}). For a $C^*$-algebra $C$, Archbold, Somerset and Timoney established the compact-stabilization identity in the completed case and showed that the stabilized constant depends only on the homeomorphism class of $\Prim(C)$ \cite[Theorem~6.2 and Corollary~6.3]{ArchboldSomersetTimoney2009}. Our result extends this identity to every finite tensor length and from $C^*$-algebras to imprimitivity bimodules.
	
	The coefficient map $\Theta_X$ is closely related to the elementary operators on Hilbert $C^*$-modules studied by Aramba\v{s}i\'c and the author \cite{ArambasicGogic2020}. For a nonzero right Hilbert $A$-module $Y$, such operators have left coefficients in $\mathcal B_A(Y)$, the $C^*$-algebra of adjointable $A$-module operators, and right coefficients in $\M(A)$. In \cite[Theorem~3.12]{ArambasicGogic2020}, we proved that the associated coefficient map is isometric if and only if $A$ is prime, extending Mathieu's corresponding result for $C^*$-algebras. If $X={}_AX_B$ is viewed as a right Hilbert $B$-module, then $\mathcal B_B(X)\cong\M(A)$ \cite[Proposition~8.1.16]{BlecherLeMerdy2004}. Thus the Hilbert-module coefficient map uses coefficients in $\M(A)$ and $\M(B)$, whereas $\Theta_X$ uses coefficients in $A$ and $B$. Module Haagerup tensor products realize the equivalence functors for strong Morita equivalence \cite[Theorem~1.2]{Blecher2001}; see also \cite{BlecherMuhlyPaulsen2000} and, for related operator-bimodule methods and Morita invariance of maximal $C^*$-algebras of quotients, \cite{AraMathieuOrtega2009}.
	
	The paper is organized as follows. Section~\ref{sec:preliminaries} fixes the ideal-theoretic, operator-space, operator-module and Morita-theoretic conventions, recalls the single-algebra results needed below, and proves the complete-isometry result in Proposition~\ref{prop:complete-isometry}. Section~\ref{sec:main} establishes the full-corner formulas, develops the finite-dimensional coordinate transfer, and proves Theorem~\ref{thm:main}. Section~\ref{sec:constants} compares the stabilized inverse constants, proves the compact-stabilization result, and concludes with remarks on ternary rings of operators and Hilbert $C^*$-modules, as well as on the limitation concerning multiplier coefficients in the nonunital case.
	
	\section{Preliminaries}\label{sec:preliminaries}
	
	This section fixes the notation and background used in the proofs. Tensor products of $C^*$-algebras without a subscript are minimal (spatial).
	
	\subsection{Ideals and multiplier centres}\label{subsec:cstar-facts}
	
	We begin with the notions of primitive, Glimm, prime and $n$-primal ideals, together with the scalar action of the multiplier centre on Glimm quotients. We also record the orthogonal-lifting lemma used in the necessary implications of Theorem~\ref{thm:main}.
	
	For a $C^*$-algebra $C$, we write $\Id(C)$ for the ideal lattice of $C$. If $n\in\N$ and $I_1,\ldots,I_n\in\Id(C)$, we use 
	\[
	I_1\cdots I_n:=\cspan\{a_1\cdots a_n:a_j\in I_j,\ 1\le j\le n\}.
	\]
	As a standard consequence of the Cohen--Hewitt factorization theorem \cite[Proposition~2.33]{RaeburnWilliams1998},
	\begin{equation}\label{eq:ideal-products}
		I_1\cdots I_n=I_1\cap\cdots\cap I_n.
	\end{equation}
	
	A \emph{primitive ideal} of $C$ is the kernel of a nonzero irreducible representation of $C$, and $\Prim(C)$ denotes the \emph{primitive ideal space} of $C$, equipped with its Jacobson (hull-kernel) topology. Two primitive ideals are \emph{Glimm equivalent} if every bounded continuous function $\Prim(C)\to\C$ takes the same value at them. A \emph{Glimm ideal} is the intersection of a Glimm-equivalence class; write $\Glimm(C)$ for the set of
	Glimm ideals.
	
	By the Dauns--Hofmann theorem \cite[Theorem~A.34]{RaeburnWilliams1998} (see also the original source \cite{DaunsHofmann1968}), the multiplier centre $Z_C$ is canonically identified with the algebra $C_b(\Prim(C))$ of bounded continuous complex-valued functions. If $z\in Z_C$ and $\widehat z\in C_b(\Prim(C))$ is the corresponding function, then
	\[
	zc+P=\widehat z(P)(c+P),\qquad c\in C,\quad P\in\Prim(C).
	\]
	For $G\in\Glimm(C)$, define $\chi_G:Z_C\to\C$ by $\chi_G(z):=\widehat z(P)$, where $P$ belongs to the Glimm-equivalence class whose intersection is $G$. This is well defined, and
	\begin{equation}\label{eq:scalar-centre}
		zc-\chi_G(z)c\in G,\qquad z\in Z_C, \quad  c\in C.
	\end{equation}
	A $C^*$-algebra $C$ is \emph{prime} if it is nonzero and the product of any two nonzero ideals of $C$ is nonzero. A proper ideal $Q\in\Id(C)$ is \emph{prime} if, for all $I_1,I_2\in\Id(C)$,
	\[
	I_1I_2\subseteq Q \qquad\implies\qquad
	I_1\subseteq Q \quad\text{or}\quad I_2\subseteq Q.
	\] 
	This condition is equivalent to $C/Q$ being a prime $C^*$-algebra. Every primitive ideal $P\in\Prim(C)$ is prime \cite[Proposition~A.17(b)]{RaeburnWilliams1998}.
	
	For $n\in\N$ with $n\geq2$, a proper ideal $Q\in\Id(C)$ is \emph{$n$-primal} if, for all $I_1,\ldots,I_n\in\Id(C)$,
	\[
	I_1\cdots I_n=0 \qquad\implies\qquad
	I_j\subseteq Q \quad\text{for some }1\le j\le n.
	\] 
	The ideal $Q$ is \emph{primal} if it is $n$-primal for every $n\geq2$. Also, for integers $2\le k\le n$, $n$-primality implies $k$-primality: append
	$n-k$ copies of $C$ to a zero-product family of $k$ ideals.
	
	Positive elements of a $C^*$-algebra are called \emph{orthogonal} when their product is zero. We shall use the following orthogonal-lifting result, in which the lifts can be chosen in prescribed ideals. Its proof adapts the construction used in \cite[p.~415, proof of Theorem~7]{ArchboldSomersetTimoney2005} and \cite[p.~1412, proof of Theorem~5.1]{ArchboldSomersetTimoney2009}, based on \cite[Proposition~2.6]{AkemannPedersen1977}.
	
	\begin{lemma}\label{lem:orthogonal-lifting}
		Let $C$ be a $C^*$-algebra, let $N\in\N$, and let $J,I_1,\ldots,I_N\in\Id(C)$. Suppose that $\bar c_j\in(I_j+J)/J$, $1\le j\le N$, are pairwise orthogonal positive elements of norm one. Then there are pairwise orthogonal positive elements $c_j\in I_j$ of norm one such that $c_j+J=\bar c_j$ for $1\le j\le N$.
	\end{lemma}
	
	\begin{proof}
		Let $\pi:C\to C/J$ be the quotient map. Since $(I_j+J)/J=\pi(I_j)$ is an ideal of $C/J$, we have $\bar c_j^{1/3}\in(I_j+J)/J$ for $1\le j\le N$. For each $1\le j\le N$, choose a self-adjoint element $x_j\in I_j$ such that $\pi(x_j)=\bar c_j^{1/3}$. Define a function 
		\[
		\tau:\mathbb R\to[0,1], \qquad \tau(t):=\min\{\max\{t,0\},1\},
		\]
		and put $g_j:=\tau(x_j)$ for $1\le j\le N$. Since $\tau(0)=0$, continuous functional calculus shows that $g_j$ is a positive contraction in $I_j$, and $\pi(g_j)=\bar c_j^{1/3}$ because $\bar c_j^{1/3}$ is a positive contraction.
		
		Since the elements $\bar c_1^{1/3},\ldots,\bar c_N^{1/3}$ are pairwise orthogonal, for distinct $i,j\in\{1,\ldots,N\}$ we have
		\[
		\pi(g_i g_j)=\bar c_i^{1/3}\bar c_j^{1/3}=0.
		\]
		Thus $g_i g_j\in J$ whenever $i\ne j$. By \cite[Proposition~2.6]{AkemannPedersen1977}, there exist $a_1,\ldots,a_N\in J$ such that the elements $h_j:=g_j-a_j\in C$, $1\le j\le N$, are positive and pairwise orthogonal. Since $a_j\in J$, we also have $\pi(h_j)=\bar c_j^{1/3}$ for $1\le j\le N$. Set
		\[
		b_j:=h_jg_jh_j\in I_j,\qquad 1\le j\le N.
		\]
		Each $b_j$ is positive. For distinct $1\le j,k\le N$, the equality $h_jh_k=0$ gives $b_jb_k=0$, and $\pi(b_j)=\bar c_j$ for every $1\le j\le N$.
		
		For $1\le j\le N$, put $c_j:=\tau(b_j)$. Since $\tau(0)=0$, continuous functional calculus shows that $c_j$ is a positive contraction in $I_j$. Moreover, $b_jb_k=0$ for $j\ne k$ and $\tau(0)=0$ imply $c_jc_k=0$, so the elements $c_1,\ldots,c_N$ are pairwise orthogonal. Since $\bar c_j$ is a positive contraction,
		\[
		\pi(c_j)=\tau(\pi(b_j))=\tau(\bar c_j)=\bar c_j,\qquad 1\le j\le N.
		\]
		Thus $\|c_j\|\geq\|\pi(c_j)\|=\|\bar c_j\|=1$, while $c_j$ is a positive contraction and hence $\|c_j\|\le1$. Therefore $\|c_j\|=1$ for every $1\le j\le N$.
	\end{proof}
	
	\subsection{Operator spaces, modules and Haagerup tensor products}\label{subsec:operator-modules}
	
	We recall the operator-space and operator-module facts used below; our terminology follows \cite[Sections~1.2, 1.5, 3.1 and~3.4]{BlecherLeMerdy2004}.
	
	For Hilbert spaces $H$ and $K$, let $\mathcal B(H,K)$ denote the space of bounded linear operators from $H$ to $K$, and write $\mathcal B(H):=\mathcal B(H,H)$. By an \emph{operator space} we mean a norm-closed linear subspace $V\subseteq\mathcal B(H,K)$, for some Hilbert spaces $H$ and $K$, with the inherited matrix norms. Thus, for $r,s\in\N$, $M_{r,s}(V)$ is canonically identified with a subspace of $\mathcal B(H^{\oplus s},K^{\oplus r})$ and carries the corresponding operator norm. We set $M_n(V):=M_{n,n}(V)$ for $n\in\N$. Write $I_n\in M_n(\C)$ for the identity matrix and $E_{ij}(v)\in M_{r,s}(V)$ for the matrix with entry $v\in V$ in position $(i,j)$ and zero elsewhere, where $1\le i\le r$ and $1\le j\le s$.
	
	For operator spaces $V,W$, a linear map $T:V\to W$ and $n\in\N$, define the \emph{$n$th amplification} of $T$ by
	\[
	T^{[n]}:M_n(V)\to M_n(W),\qquad
	T^{[n]}([v_{ij}]):=[T(v_{ij})].
	\]
	The map $T$ is \emph{completely bounded} if
	\[
	\|T\|_{\cb}:=\sup_{n\in\N}\|T^{[n]}\|<\infty,
	\]
	\emph{completely contractive} if $\|T\|_{\cb}\le1$, and \emph{completely isometric} if every $T^{[n]}$ is isometric. The space $\CB(V,W)$ of completely bounded linear maps from $V$ to $W$ is a Banach space with respect to $\|\cdot\|_{\cb}$ and carries the canonical operator-space structure characterized, for each $n\in\N$, by the complete isometry
	\begin{equation}\label{eq:cb-matrix-identification}
		M_n(\CB(V,W))\cong\CB(V,M_n(W)),\qquad
		[T_{ij}]\mapsto\left(v\mapsto[T_{ij}(v)]\right).
	\end{equation}
	See \cite[1.2.19]{BlecherLeMerdy2004}. We set $\CB(V):=\CB(V,V)$.
	
	Every $C^*$-algebra carries its canonical operator-space structure. A homomorphism between $C^*$-algebras is completely contractive if and only if it is a $*$-homomorphism; in this case it is completely isometric if and only if it is injective \cite[1.2.3 and Proposition~1.2.4]{BlecherLeMerdy2004}.
	
	If $V_0\subseteq V$ is a closed linear subspace, then $V/V_0$ carries the \emph{quotient operator-space structure} determined by
	$M_n(V/V_0)\cong M_n(V)/M_n(V_0)$; in particular, the canonical quotient map $V\to V/V_0$ is completely contractive. See \cite[1.2.14]{BlecherLeMerdy2004}.
	
	For $n\in\N$ and $U=[u_{ij}]\in M_n(V\odot W)$, define the \emph{Haagerup norm} of $U$ by
	\[
	\|U\|_h:=
	\inf\left\{\|R\|\,\|S\|:
	s\in\N,\ R\in M_{n,s}(V),\
	S\in M_{s,n}(W),\ U=R\odot S\right\}.
	\]
	Here, if $R=[v_{ik}]\in M_{n,s}(V)$ and
	$S=[w_{kj}]\in M_{s,n}(W)$, then
	$R\odot S\in M_n(V\odot W)$ has $(i,j)$-entry
	\[
	(R\odot S)_{ij}:=
	\sum_{k=1}^s v_{ik}\otimes w_{kj}.
	\]
	The completion is the \emph{Haagerup tensor product}
	$V\otimes_hW$ \cite[1.5.4]{BlecherLeMerdy2004}. At the first matrix level this recovers \eqref{eq:haagerup-norm} when $V=W=\M(C)$. If $V'$ and $W'$ are operator spaces and $\alpha:V\to V'$ and $\beta:W\to W'$ are complete contractions, then $\alpha\otimes_h\beta$ is a complete contraction \cite[1.5.5]{BlecherLeMerdy2004}; if both maps are completely isometric, Haagerup injectivity \cite[Theorem~3.6]{BlecherPaulsen1991} gives
	\begin{equation}\label{eq:haagerup-injectivity}
		\|(\alpha\otimes_h\beta)^{[n]}(U)\|_h
		=\|U\|_h,
		\qquad n\in\N,\quad U\in M_n(V\otimes_hW).
	\end{equation}
	Let $A$ and $B$ be $C^*$-algebras. An \emph{operator $A$--$B$ bimodule} is an $A$--$B$ bimodule $X$ which is also an operator space and for which there exist Hilbert spaces $H$ and $K$, $*$-homomorphisms $\pi:A\to\mathcal B(K)$ and $\sigma:B\to\mathcal B(H)$, and a linear complete isometry $\Phi:X\to\mathcal B(H,K)$ such that
	\[
	\Phi(ax)=\pi(a)\Phi(x),\qquad \Phi(xb)=\Phi(x)\sigma(b),
	\qquad a\in A,\quad b\in B,\quad x\in X.
	\]
	Thus the module actions are realized by operator multiplication; see \cite[3.1.1]{BlecherLeMerdy2004}.
	
	Let $Z$ be a unital commutative $C^*$-algebra, and let $V$ and $W$ be right and left operator $Z$-modules, respectively. Their \emph{module Haagerup tensor product} is
	\begin{align*}
		V\otimes_{Z,h}W&:=(V\otimes_hW)/N_Z(V,W),\\
		N_Z(V,W)&:=
		\cspan\{vz\otimes w-v\otimes zw:
		v\in V,\ w\in W,\ z\in Z\};
	\end{align*}
	see \cite[3.4.2]{BlecherLeMerdy2004}. We write $v\otimes_Z w$ for the image of $v\otimes w$ in this quotient, $V\odot_Z W$ for the image of $V\odot W$ in $V\otimes_{Z,h}W$, and $\|\cdot\|_{Z,h}$ for the quotient norm; thus
	\[
	vz\otimes_Z w=v\otimes_Z zw,
	\qquad v\in V,\quad w\in W,\quad z\in Z.
	\]
	Suppose that $A$ and $B$ are, respectively, right and left operator $Z$-modules, and let $X$ be an operator $A$--$B$ bimodule. By \cite[3.1.3 and Proposition~3.1.14]{BlecherLeMerdy2004}, the canonical left and right module actions on the matrix spaces over $X$ are contractive. Together with \eqref{eq:cb-matrix-identification} and the defining factorization of the Haagerup norm, this shows that the coefficient action induces a complete contraction
	\[
	A\otimes_hB\to\CB(X),\qquad
	a\otimes b\mapsto(x\mapsto axb).
	\]
	If the coefficient action is \emph{$Z$-balanced}, that is,
	\[
	(az)xb=ax(zb),
	\qquad a\in A,\quad b\in B,\quad x\in X,\quad z\in Z,
	\]
	then this map annihilates $N_Z(A,B)$ and hence, by \cite[3.4.2]{BlecherLeMerdy2004}, induces the complete contraction
	\begin{equation}\label{eq:module-coefficient-map}
		\Theta_X:A\otimes_{Z,h}B\to\CB(X),\qquad
		\Theta_X(a\otimes_Z b)(x):=axb,
	\end{equation}
	where $a\in A$, $b\in B$ and $x\in X$.
	
	We shall also use functoriality of the module Haagerup tensor product. Let $V'$ and $W'$ be right and left operator $Z$-modules, respectively. If $\alpha:V\to V'$ and $\beta:W\to W'$ are completely contractive $Z$-module maps, then \cite[Lemma~3.4.5]{BlecherLeMerdy2004} gives
	\begin{equation}\label{eq:module-haagerup-functoriality}
		\alpha\otimes_{Z,h}\beta:
		V\otimes_{Z,h}W\to V'\otimes_{Z,h}W', \qquad \|\alpha\otimes_{Z,h}\beta\|_{\cb}\le1.
	\end{equation}
	If, in addition, $\alpha$ and $\beta$ have completely contractive $Z$-module left inverses, then applying \eqref{eq:module-haagerup-functoriality} to these left inverses shows that $\alpha\otimes_{Z,h}\beta$ is completely isometric.
	
	Finally, for $n\in\N$, set 
	\[
	C_n(V):=M_{n,1}(V) \qquad \text{and} \qquad R_n(W):=M_{1,n}(W),
	\]
	with their induced entrywise $Z$-actions. For $R=[v_i]_{i=1}^n\in C_n(V)$ and $S=[w_j]_{j=1}^n\in R_n(W)$, the module column--row identification \cite[3.4.11]{BlecherLeMerdy2004} gives the complete isometry
	\begin{equation}\label{eq:matrix-haagerup}
		C_n(V)\otimes_{Z,h}R_n(W)
		\cong M_n(V\otimes_{Z,h}W),\qquad
		R\otimes_Z S\mapsto[v_i\otimes_Z w_j]_{i,j=1}^n.
	\end{equation}
	
	\subsection{Morita equivalence, ideals and centres}\label{subsec:ideals}
	
	We now recall the Morita-theoretic facts used below: the linking-algebra realization of an imprimitivity bimodule, the Rieffel correspondence for ideals, and the induced identification of multiplier centres.
	
	We use the Hilbert-module conventions of \cite[Chapter~1]{Lance1995} and the imprimitivity-bimodule conventions of \cite[Sections~3.1--3.2]{RaeburnWilliams1998}. For $C^*$-algebras $A$ and $B$, an $A$--$B$ \emph{imprimitivity bimodule} $X={}_AX_B$ is a full right Hilbert $B$-module, with inner product linear in the second variable, whose left action identifies $A$ with the $C^*$-algebra of \emph{compact operators} on $X$. More precisely, for $x,y\in X$, define
	\[
	\vartheta_{x,y}:X\to X,\qquad \vartheta_{x,y}(v):=x\langle y,v\rangle_B,\qquad v\in X,
	\]
	and set
	\[
	\mathcal K_B(X):=\cspan\{\vartheta_{x,y}:x,y\in X\},
	\]
	where the closure is taken in the operator norm. Thus the left action induces a $*$-isomorphism of $A$ onto $\mathcal K_B(X)$. Here \emph{fullness} means $\cspan\langle X,X\rangle_B=B$, and, under this identification, the left $A$-valued inner product is given by ${}_A\langle x,y\rangle:=\vartheta_{x,y}$ for $x,y\in X$. We call $A$ and $B$ the \emph{coefficient algebras} of $X$. The $C^*$-algebras $A$ and $B$ are said to be \emph{strongly Morita equivalent} if there exists an $A$--$B$ imprimitivity bimodule $X$ \cite{Rieffel1974}. The \emph{linking algebra} of $X$ is the $C^*$-algebra
	\begin{equation}\label{eq:linking-algebra}
		D:=\Link(X):=\begin{bmatrix}A&X\\\widetilde X&B\end{bmatrix},
	\end{equation}
	where $\widetilde X$ is the conjugate $B$--$A$ bimodule, and $x^*\in\widetilde X$ denotes the element corresponding to $x\in X$ under the canonical conjugate-linear bijection $X\to\widetilde X$. The products of opposite off-diagonal entries are $xy^*:={}_A\langle x,y\rangle$ and $x^*y:=\langle x,y\rangle_B$ for $x,y\in X$. The $C^*$-norm on $D$ is the operator norm induced by its canonical action on the Hilbert $B$-module $X\oplus B$; see \cite[8.1.17]{BlecherLeMerdy2004}. Let $p,q\in\M(D)$ be the canonical linking projections; then
	\[
	p+q=1_{\M(D)},\qquad pq=0,\qquad A=pDp,\qquad X=pDq\qquad\text{and}\qquad B=qDq.
	\]
	The canonical operator-space structure of $X$ is the one inherited from $D$ \cite[8.2.1]{BlecherLeMerdy2004}. With this structure, $X$ is an operator $\M(A)$--$\M(B)$ bimodule, and hence in particular an operator $A$--$B$ bimodule \cite[8.2.3(1)]{BlecherLeMerdy2004}. For $n\in\N$, the matrix bimodule $M_n(X)$ is an $M_n(A)$--$M_n(B)$ imprimitivity bimodule with inner products
	\[
	{}_{M_n(A)}\langle U,V\rangle:=UV^*
	\qquad\text{and}\qquad
	\langle U,V\rangle_{M_n(B)}:=U^*V,
	\qquad U,V\in M_n(X).
	\]
	Rearranging the block coordinates identifies the linking algebra of $M_n(X)$ with $M_n(D)$.
	
	For a $C^*$-algebra $C$, a projection $e\in\M(C)$ is called \emph{full} if $\cspan(CeC)=C$; in this case $eCe$ is a \emph{full corner} of $C$ \cite[p.~50]{RaeburnWilliams1998}. Both canonical linking projections are full. If $e,f\in\M(C)$ are full projections, then $eCf$ is an $eCe$--$fCf$ imprimitivity bimodule, with inner products 
	\[
	{}_{eCe}\langle x,y\rangle:=xy^* \qquad  \text{and} \qquad  \langle x,y\rangle_{fCf}:=x^*y, \qquad x,y\in eCf,
	\]
	and with the matrix norms inherited from $C$. We call $eCf$ a \emph{full rectangular corner} of $C$.
	
	Given an imprimitivity bimodule $X={}_AX_B$ and $I\in\Id(A)$, put
	\[
	IX:=\cspan\{ax:a\in I,\ x\in X\} \subseteq X \qquad \text{and} \qquad \widetilde{IX}:=\{x^*:x\in IX\}\subseteq\widetilde X.
	\]
	The \emph{Rieffel correspondent} $I^X\in\Id(B)$ of $I$ and its \emph{linking ideal} $\widehat I\in\Id(D)$ are given, respectively, by
	\[
	I^X:=\cspan\{\langle x,ay\rangle_B:a\in I,\ x,y\in X\} \qquad \text{and} \qquad
	\widehat I:=
	\begin{bmatrix}
		I & IX\\
		\widetilde{IX} & I^X
	\end{bmatrix}.
	\]
	The \emph{Rieffel correspondence} \cite[Theorem~3.22]{RaeburnWilliams1998} is the order isomorphism
	\begin{equation}\label{eq:rieffel-correspondence}
		\Id(A)\to\Id(B),\qquad I\mapsto I^X.
	\end{equation}
	Its restriction $P\mapsto P^X$ is a homeomorphism from $\Prim(A)$ onto $\Prim(B)$ \cite[Corollary~3.33]{RaeburnWilliams1998}.
	
	If $C$ is a $C^*$-algebra and $e\in\M(C)$ is a full projection, then $Ce$ is a $C$--$eCe$ imprimitivity bimodule \cite[Example~3.6]{RaeburnWilliams1998}. Its Rieffel correspondence is the order isomorphism $J\mapsto eJe$ from $\Id(C)$ onto $\Id(eCe)$, with inverse
	$I\mapsto\cspan(CIC)$ for $I\in\Id(eCe)$. In particular, $Dp$ is a $D$--$A$ imprimitivity bimodule, and the inverse correspondence is $I\mapsto\widehat I$ from $\Id(A)$ onto $\Id(D)$; it restricts to a homeomorphism from $\Prim(A)$ onto $\Prim(D)$.
	
	\begin{lemma}\label{lem:morita-ideals}
		Let $X={}_AX_B$ be an imprimitivity bimodule. For every integer $n\geq2$, the Rieffel correspondence \eqref{eq:rieffel-correspondence} preserves Glimm ideals and $n$-primality in both directions. In particular, if $C$ is a $C^*$-algebra, $e\in\M(C)$ is a full projection and $k\in\N$, the same holds for the correspondences
		\[
		\Id(C)\to\Id(eCe),\qquad J\mapsto eJe,
		\qquad\text{and}\qquad
		\Id(C)\to\Id(M_k(C)),\qquad I\mapsto M_k(I).
		\]
	\end{lemma}
	
	\begin{proof}
		Since \eqref{eq:rieffel-correspondence} is an order isomorphism, it preserves arbitrary intersections, the zero ideal and proper ideals. Hence \eqref{eq:ideal-products} gives preservation of $n$-primality in both directions. The induced homeomorphism of primitive ideal spaces carries Glimm-equivalence classes onto Glimm-equivalence classes, and preservation of arbitrary intersections therefore gives preservation of Glimm ideals.
		
		The full-corner assertion follows by applying this to the $C$--$eCe$ imprimitivity bimodule $Ce$. For $C\ne0$, the matrix correspondence is the inverse of the full-corner correspondence associated with the full projection $E_{11}(1_{\M(C)})\in\M(M_k(C))$; for $C=0$ the assertion is immediate.
	\end{proof}
	
	For a $C^*$-algebra $C$ and a full projection $e\in\M(C)$, compression gives a $*$-isomorphism
	\begin{equation}\label{eq:corner-centre}
		Z_C\to Z(\M(eCe)),\qquad z\mapsto eze.
	\end{equation}
	Indeed, for $z\in Z_C$ and $P\in\Prim(C)$, $eze$ acts on $eCe/ePe$ as the scalar $\widehat z(P)$. The full-corner homeomorphism $\Prim(C)\to\Prim(eCe)$, $P\mapsto ePe$ \cite[Example~3.6 and Corollary~3.33]{RaeburnWilliams1998}, identifies the corresponding algebras of bounded continuous functions, so the assertion follows from the Dauns--Hofmann identification in Section~\ref{subsec:cstar-facts}.
	
	For the remainder of this subsection, fix a nonzero imprimitivity bimodule $X={}_AX_B$, and retain the notation $D=\Link(X)$ and $p,q\in\M(D)$ for its linking algebra and canonical linking projections. Applying \eqref{eq:corner-centre} to $p$ and $q$ identifies $Z(\M(D))$, $Z(\M(A))$ and $Z(\M(B))$; we use the notation $Z_X$ from the Introduction for this common multiplier centre. The induced $*$-isomorphism between $Z(\M(A))$ and $Z(\M(B))$ is characterized by
	\[
	zx=xz,\qquad z\in Z_X,\quad x\in X;
	\]
	see \cite[Corollary~8.1.20]{BlecherLeMerdy2004}. Thus $A$ and $B$ are, respectively, right and left operator $Z_X$-modules, and the coefficient action on $X$ is $Z_X$-balanced. Hence \eqref{eq:module-coefficient-map} gives precisely the coefficient map \eqref{eq:theta-X}; in particular, $\Theta_X$ is completely contractive. For every $n\in\N$, the map $z\mapsto zI_n$ identifies $Z_X$ with the common multiplier centre of $M_n(A)$ and $M_n(B)$.
	
	Let $I\in\Id(A)$ contain a Glimm ideal $G\in\Glimm(A)$. By Lemma~\ref{lem:morita-ideals}, $\widehat G\in\Glimm(D)$. Since $p\widehat Gp=G$ and $q\widehat Gq=G^X$, \eqref{eq:scalar-centre} applied in $D$ to $\widehat G$ shows that every $z\in Z_X$ acts on both $A/G$ and $B/G^X$ as multiplication by the same scalar $\chi_{\widehat G}(z)$. Since $G\subseteq I$ and $G^X\subseteq I^X$, the same is true on $A/I$ and $B/I^X$, and the quotient maps $A\to A/I$ and $B\to B/I^X$ are $Z_X$-module maps. Hence, for $a\in A$, $b\in B$ and $z\in Z_X$,
	\begin{align*}
		(az+I)\otimes(b+I^X)
		&\stackrel{\eqref{eq:scalar-centre}}{=}
		\chi_{\widehat G}(z)(a+I)\otimes(b+I^X)\\
		&= (a+I)\otimes\chi_{\widehat G}(z)(b+I^X)\\
		&\stackrel{\eqref{eq:scalar-centre}}{=}
		(a+I)\otimes(zb+I^X).
	\end{align*}
	Thus the balancing subspace for $(A/I)\otimes_{Z_X,h}(B/I^X)$ is zero, so this module Haagerup tensor product is canonically $(A/I)\otimes_h(B/I^X)$. Therefore
	\eqref{eq:module-haagerup-functoriality} gives a complete contraction
	\begin{equation}\label{eq:balanced-quotient-X}
		\begin{aligned}
			q_I&:A\otimes_{Z_X,h}B\to(A/I)\otimes_h(B/I^X),
			\qquad u\mapsto u_I,\\
			(a\otimes_{Z_X}b)_I&:=(a+I)\otimes(b+I^X),
			\qquad a\in A,\ b\in B.
		\end{aligned}
	\end{equation}
	Since every primitive ideal contains the Glimm ideal of its class, $q_P$ is defined for every $P\in\Prim(A)$.
	
	\subsection{The single-algebra results}\label{subsec:single-algebra}
	
	This subsection records the single-algebra norm formulas and ideal-theoretic criteria that will be transferred to imprimitivity bimodules in Section~\ref{sec:main}. We also isolate the matrix-level argument needed to pass from isometry to complete isometry.
	
	For a nonzero $C^*$-algebra $C$, we regard $C$ as a $C$--$C$ imprimitivity bimodule with left and right multiplication and inner products
	\[
	{}_C\langle a,b\rangle:=ab^* \qquad\text{and}\qquad  \langle a,b\rangle_C:=a^*b,
	\qquad a,b\in C.
	\]
	For this bimodule, the Rieffel correspondence is the identity on $\Id(C)$, and the coefficient map from \eqref{eq:module-coefficient-map}, with $Z=Z_C$, is $\Theta_C=\theta_C$. We now recall the results of Somerset and Archbold--Somerset--Timoney for $\theta_C$.
	
	\begin{theorem}[Somerset; Archbold--Somerset--Timoney]\label{thm:diagonal}
		Let $C$ be a nonzero $C^*$-algebra. For every $u\in C\otimes_{Z_C,h}C$,
		\begin{equation}\label{eq:diagonal-quotients}
			\|u\|_{Z_C,h}=\sup_{G\in\Glimm(C)}\|u_G\|_h \qquad \text{and}
			\qquad \|\theta_C(u)\|_{\cb}=\sup_{P\in\Prim(C)}\|u_P\|_h.
		\end{equation}
		Moreover:
		\begin{enumerate}[label=\textup{(\roman*)}]
			\item $\theta_C$ is injective if and only if every Glimm ideal of $C$ is $2$-primal;
			\item for each $\ell\in\N$, $\theta_C$ preserves the norm of every $u\in C\odot_{Z_C}C$ with $\len(u)\le\ell$ if and only if every Glimm ideal of $C$ is $(\ell^2+1)$-primal;
			\item $\theta_C$ is isometric if and only if every Glimm ideal of $C$ is primal.
		\end{enumerate}
	\end{theorem}
	
	The norm formulas are \cite[Theorem~3.5 and Proposition~3.6]{ArchboldSomersetTimoney2009}.
	There, $u_I$ denotes the image of $u$ in $(C\otimes_hC)/(C\otimes_hI+I\otimes_hC)$, which is canonically isometrically isomorphic to $(C/I)\otimes_h(C/I)$ by \cite[Corollary~2.6]{AllenSinclairSmith1993}; this gives \eqref{eq:diagonal-quotients} in the notation used here. The primitive-ideal formula extends from algebraic tensors by continuity, since both sides are $1$-Lipschitz with respect to $\|\cdot\|_{Z_C,h}$. Assertion \textup{(i)} is \cite[Theorem~3.8(i)]{ArchboldSomersetTimoney2009}, extending \cite[Corollary~6(ii)]{Somerset1998} to the nonunital case. Assertion \textup{(ii)} is \cite[Theorem~3.9]{ArchboldSomersetTimoney2009}, extending \cite[Theorem~17]{ArchboldSomersetTimoney2005}. The isometry criterion in \textup{(iii)} is \cite[Theorem~3.8(ii)]{ArchboldSomersetTimoney2009}; in the unital case, sufficiency is \cite[Theorem~4]{Somerset1998}, while necessity follows from \cite[Theorem~7]{ArchboldSomersetTimoney2005}. Finally, adding zero summands allows the $\ell$-term statements to be applied to tensors of length at most $\ell$.
	
	\begin{proposition}\label{prop:complete-isometry}
		Let $C$ be a nonzero $C^*$-algebra. If $\theta_C$ is isometric, then $\theta_C$ is completely isometric.
	\end{proposition}
	
	\begin{proof}
		Suppose that $\theta_C$ is isometric and fix $n\in\N$. By Theorem~\ref{thm:diagonal}\textup{(iii)}, every Glimm ideal of $C$ is primal. By Lemma~\ref{lem:morita-ideals}, the same is true of $M_n(C)$, so $\theta_{M_n(C)}$ is isometric by Theorem~\ref{thm:diagonal}\textup{(iii)} applied to $M_n(C)$. By \eqref{eq:corner-centre}, we identify the centre of $\M(M_n(C))$ with $Z_C$ through $z\mapsto zI_n$, $z\in Z_C$.
		
		By the module column--row identification \eqref{eq:matrix-haagerup}, $M_n(C\otimes_{Z_C,h}C)\cong C_n(C)\otimes_{Z_C,h}R_n(C)$ completely isometrically. Embedding $C_n(C)$ as the first column of $M_n(C)$ and $R_n(C)$ as the first row gives completely isometric $Z_C$-module maps, while taking the first column and first row gives completely contractive $Z_C$-module left inverses. Hence \eqref{eq:module-haagerup-functoriality} and the left-inverse criterion following it give a complete isometry
		\[
		\iota_n:M_n(C\otimes_{Z_C,h}C) \to M_n(C)\otimes_{Z_C,h}M_n(C),\qquad 
		\iota_n(E_{ij}(a\otimes_{Z_C}b)):=E_{i1}(a)\otimes_{Z_C}E_{1j}(b),
		\]
		for $a,b\in C$ and $1\le i,j\le n$.
		
		Define a map $\kappa_n:M_n(C)\to C$ by $\kappa_n([c_{ij}]):=c_{11}$. For $U\in M_n(C\otimes_{Z_C,h}C)$, let $T_U:C\to M_n(C)$ denote the map corresponding to $\theta_C^{[n]}(U)\in M_n(\CB(C))$ under \eqref{eq:cb-matrix-identification}. For $a,b\in C$, $1\le i,j\le n$, and $W=[w_{rs}]\in M_n(C)$, we have
		\begin{align*}
			\theta_{M_n(C)}
			\left(\iota_n(E_{ij}(a\otimes_{Z_C}b))\right)(W)  &=E_{i1}(a)WE_{1j}(b)=E_{ij}(aw_{11}b)\\
			&=T_{E_{ij}(a\otimes_{Z_C}b)}(\kappa_n(W)).
		\end{align*}
		Since such elementary matrix tensors linearly span a dense subspace of $M_n(C\otimes_{Z_C,h}C)$, linearity and continuity give
		\begin{equation}\label{eq:known-amplification}
			\theta_{M_n(C)}(\iota_n(U))=T_U\circ\kappa_n.
		\end{equation}
		The complete contraction $\kappa_n$ has the completely isometric right inverse $C\to M_n(C)$, $c\mapsto E_{11}(c)$. Thus composition with $\kappa_n$ preserves completely bounded norms, and
		\[
		\|\theta_C^{[n]}(U)\|
		\stackrel{\eqref{eq:cb-matrix-identification}}{=}
		\|T_U\|_{\cb}
		=\|T_U\circ\kappa_n\|_{\cb}
		\stackrel{\eqref{eq:known-amplification}}{=}
		\|\theta_{M_n(C)}(\iota_n(U))\|_{\cb}
		=\|\iota_n(U)\|=\|U\|.
		\]
		Since $n\in\N$ was arbitrary, $\theta_C$ is completely isometric.
	\end{proof}
	
	\section{Proof of Theorem~\ref{thm:main}}\label{sec:main}
	
	We prove the sufficient implications in Theorem~\ref{thm:main} by applying the single-algebra results to the linking algebra. For the necessary implications, we construct tensors in $A\odot_{Z_X}B$ witnessing failure of injectivity or norm preservation when the corresponding Glimm-ideal condition fails.
	
	\subsection{Full corners and quotient formulas}
	
	We first show that full rectangular corners preserve the relevant tensor norms, algebraic tensor lengths and coefficient-map norms, and establish the corresponding quotient identities. Applied to the linking algebra, these identities give the Glimm- and primitive-quotient formulas for imprimitivity bimodules; for an analogous argument, see \cite[Lemmas~3.7 and~3.11]{ArambasicGogic2020}.
	
	\begin{lemma}\label{lem:corner-calculus}
		Let $C$ be a nonzero $C^*$-algebra and $e,f\in\M(C)$ full projections. Put $E:=eCe$, $F:=fCf$ and $Y:=eCf$, and identify the multiplier centres of $E$ and $F$ with $Z:=Z_C$ through \eqref{eq:corner-centre}. The corner inclusions induce a complete isometry $j:E\otimes_{Z,h}F\to C\otimes_{Z,h}C$ preserving algebraic tensor length. The coefficient action on $Y$ is $Z$-balanced, so \eqref{eq:module-coefficient-map} gives a complete contraction $\Theta_Y:E\otimes_{Z,h}F\to\CB(Y)$, and
		\begin{equation}\label{eq:corner-action-norm}
			\|\Theta_Y^{[n]}(U)\|=\|\theta_C^{[n]}(j^{[n]}(U))\|, \qquad n\in\N, \quad U\in M_n(E\otimes_{Z,h}F).
		\end{equation}
		If $J\in\Id(C)$ is proper and contains a Glimm ideal, then, using the quotient notation of \eqref{eq:balanced-quotient-X},
		\begin{equation}\label{eq:corner-quotient-norm}
			\|(j(u))_J\|_h=\|u_{eJe}\|_h,
			\qquad u\in E\otimes_{Z,h}F.
		\end{equation}
	\end{lemma}
	
	\begin{proof}
		The corner inclusions $E\hookrightarrow C$ and $F\hookrightarrow C$ are injective $*$-homomorphisms and hence completely isometric $Z$-module maps. They have the completely contractive $Z$-module left inverses $C\to E$, $c\mapsto ece$, and $C\to F$, $c\mapsto fcf$, respectively. Hence \eqref{eq:module-haagerup-functoriality} and the left-inverse criterion following it show that the induced map $j$ is completely isometric. Both $j$ and the tensor map induced by these left inverses send elementary tensors to elementary tensors. For $u\in E\odot_Z F$, applying these maps to finite representations of $u$ and $j(u)$, respectively, gives $\len(j(u))\le\len(u)$ and $\len(u)\le\len(j(u))$.
		
		Define $\kappa:C\to Y$ by $\kappa(c):=ecf$ for $c\in C$. Let $r\in\N$, $a_k\in E$ and $b_k\in F$ for $1\le k\le r$, and let $u=\sum_{k=1}^r a_k\otimes_Z b_k\in E\odot_Z F$. Then, for $c\in C$,
		\[
		\theta_C(j(u))(c)=\sum_{k=1}^r a_kcb_k=\sum_{k=1}^r a_k(ecf)b_k=\Theta_Y(u)(\kappa(c)).
		\]
		By continuity, $\theta_C(j(u))=\Theta_Y(u)\circ\kappa$ for every $u\in E\otimes_{Z,h}F$, with $\theta_C(j(u))$ regarded as a map from $C$ into $Y$. Since $\kappa(y)=y$ for $y\in Y$, it follows that $\Theta_Y(u)=\theta_C(j(u))|_Y$. For $n\in\N$ and $U\in M_n(E\otimes_{Z,h}F)$, use \eqref{eq:cb-matrix-identification} to regard $\Theta_Y^{[n]}(U)$ as a map $Y\to M_n(Y)$ and $\theta_C^{[n]}(j^{[n]}(U))$ as a map $C\to M_n(Y)$. The latter change of codomain leaves the completely bounded norm unchanged, since $Y$ has the matrix norms inherited from $C$. These maps satisfy $\theta_C^{[n]}(j^{[n]}(U))=\Theta_Y^{[n]}(U)\circ\kappa$. Since $\kappa$ is a complete contraction, it follows that
		\[
		\|\theta_C^{[n]}(j^{[n]}(U))\|
		\le \|\Theta_Y^{[n]}(U)\|.
		\]
		Conversely,
		$\Theta_Y^{[n]}(U)=\theta_C^{[n]}(j^{[n]}(U))|_Y$, and the inclusion
		$Y\hookrightarrow C$ is completely isometric, so
		\[
		\|\Theta_Y^{[n]}(U)\| \le \|\theta_C^{[n]}(j^{[n]}(U))\|.
		\]
		This proves \eqref{eq:corner-action-norm}.
		
		The ideals $eJe$ and $fJf$ correspond under $Y$. Indeed, since $e$ is full, 
		\[
		\cspan(CeJeC)=\cspan\left((CeC)J(CeC)\right)=J,
		\]
		and hence
		\[
		(eJe)^Y=\cspan(fCeJeCf)=f\,\cspan(CeJeC)\,f=fJf.
		\] 
		Since $J$ contains a Glimm ideal, both contain Glimm ideals by Lemma~\ref{lem:morita-ideals}. Hence the quotient maps are well defined by \eqref{eq:balanced-quotient-X}. The inclusions into $C/J$ induce injective $*$-homomorphisms $E/eJe\to C/J$ and $F/fJf\to C/J$, because $J\cap E=eJe$ and $J\cap F=fJf$. These maps are therefore completely isometric, so their Haagerup tensor product is completely isometric by \eqref{eq:haagerup-injectivity} and sends $u_{eJe}$ to $(j(u))_J$. This proves \eqref{eq:corner-quotient-norm}.
	\end{proof}
	
	\begin{proposition}
		For a nonzero imprimitivity bimodule $X={}_AX_B$ and every $u\in A\otimes_{Z_X,h}B$,
		\begin{equation}\label{eq:offdiag-quotients}
			\|u\|_{Z_X,h}=\sup_{G\in\Glimm(A)}\|u_G\|_h,
			\qquad \|\Theta_X(u)\|_{\cb}=\sup_{P\in\Prim(A)}\|u_P\|_h.
		\end{equation}
	\end{proposition}
	
	\begin{proof}
		Fix $u\in A\otimes_{Z_X,h}B$. Set $D:=\Link(X)$, with the canonical linking projections $p,q\in\M(D)$, and let $j:A\otimes_{Z_X,h}B\to D\otimes_{Z_D,h}D$ be the embedding of Lemma~\ref{lem:corner-calculus}; under \eqref{eq:corner-centre}, $Z_D=Z_X$. The Rieffel correspondences and Lemma~\ref{lem:morita-ideals} identify the primitive and Glimm ideals of $A$ with those of $D$ by $I\mapsto\widehat I$. If $I\in\Glimm(A)\cup\Prim(A)$, then $\widehat I$ is proper and contains a Glimm ideal of $D$, while $p\widehat I p=I$ and $q\widehat I q=I^X$. Thus \eqref{eq:corner-quotient-norm} gives $\|(j(u))_{\widehat I}\|_h=\|u_I\|_h$. Hence
		\begin{align*}
			\|u\|_{Z_X,h}&=\|j(u)\|_{Z_D,h}
			\stackrel{\eqref{eq:diagonal-quotients}}{=}
			\sup_{G\in\Glimm(A)}\|(j(u))_{\widehat G}\|_h
			\stackrel{\eqref{eq:corner-quotient-norm}}{=}
			\sup_{G\in\Glimm(A)}\|u_G\|_h,\\
			\|\Theta_X(u)\|_{\cb}&\stackrel{\eqref{eq:corner-action-norm}}{=}
			\|\theta_D(j(u))\|_{\cb}
			\stackrel{\eqref{eq:diagonal-quotients}}{=}
			\sup_{P\in\Prim(A)}\|(j(u))_{\widehat P}\|_h
			\stackrel{\eqref{eq:corner-quotient-norm}}{=}
			\sup_{P\in\Prim(A)}\|u_P\|_h.
		\end{align*}
	\end{proof}
	
	\subsection{Finite-coordinate transfer}
	
	To prove the necessity statements in Theorem~\ref{thm:main}, we need tensors in $A\odot_{Z_X}B$ that witness failure of injectivity or norm preservation, rather than tensors whose coefficients merely lie in the linking algebra. We therefore construct finite families of positive elements $c_j\in A$ and $d_j\in B$, indexed by the same set, such that for every $P\in\Prim(A)$ some index $j$ satisfies $c_j\in P$ and $d_j\in P^X$.
	
	For $N\in\N$, set $[N]:=\{1,\ldots,N\}$ and let $e_1,\ldots,e_N\in\C^N$ be the coordinate vectors, with $\C^N$ regarded as a $C^*$-algebra with the supremum norm. For $S\subseteq[N]$, let $\C^S$ be the coordinate algebra indexed by $S$, and let $\rho_S:\C^N\to\C^S$ be coordinate restriction, with $\C^\emptyset:=\{0\}$, and set 
	\[
	w_S:=(\rho_S\otimes\rho_S)(w), \qquad \text{for } w\in\C^N\odot\C^N.
	\]
	\begin{proposition}\label{prop:transport-obstruction}
		Let $X={}_AX_B$ be a nonzero imprimitivity bimodule, let $N\in\N$, $N\geq2$, and let $G\in\Glimm(A)$ be a Glimm ideal which is not $N$-primal. If $N\geq3$, assume that $G$ is $(N-1)$-primal. There exists a complete isometry $\Lambda:\C^N\otimes_h\C^N\to A\otimes_{Z_X,h}B$ with range contained in $A\odot_{Z_X}B$ such that, for every $w\in\C^N\odot\C^N$,
		\begin{equation}\label{eq:coordinate-transfer}
			\len(\Lambda(w))\le\len(w),\qquad
			\|\Theta_X(\Lambda(w))\|_{\cb}\le\max_{S\subsetneq[N]}\|w_S\|_h.
		\end{equation}
	\end{proposition}
	
	\begin{proof}
		We adapt the orthogonal-lifting argument in \cite[Theorem~7]{ArchboldSomersetTimoney2005} and \cite[Theorem~5.1]{ArchboldSomersetTimoney2009} to $A$ and $B$, using corresponding ideals under the Rieffel correspondence. Choose $I_1,\ldots,I_N\in\Id(A)$ with $I_1\cdots I_N=0$ and $I_j\not\subseteq G$ for $1\le j\le N$. For $N=2$, $I_1I_2=0$. Suppose that $N\geq3$ and that $i,j\in[N]$ are distinct. The $N-1$ ideals $I_iI_j$ and $I_k$, $k\in[N]\setminus\{i,j\}$, have zero product by \eqref{eq:ideal-products}. Since no $I_k$, $k\in[N]\setminus\{i,j\}$, is contained in $G$, $(N-1)$-primality forces $I_iI_j\subseteq G$. Thus the ideals $(I_j+G)/G$, $1\le j\le N$, are nonzero and have pairwise zero products. The Rieffel order isomorphism preserves inclusions and intersections and hence, by \eqref{eq:ideal-products}, products. Therefore $I_j^X\not\subseteq G^X$ for $1\le j\le N$ and $I_i^XI_j^X\subseteq G^X$ for distinct $i,j\in[N]$. Hence the ideals $(I_j^X+G^X)/G^X$, $1\le j\le N$, are also nonzero and have pairwise zero products.
		
		For $1\le j\le N$, choose positive norm-one elements $\bar c_j\in(I_j+G)/G$ and $\bar d_j\in(I_j^X+G^X)/G^X$. On each side these elements are pairwise orthogonal. Applying Lemma~\ref{lem:orthogonal-lifting} in $A$ and in $B$ gives positive norm-one lifts $c_j\in I_j$ and $d_j\in I_j^X$ with $c_j+G=\bar c_j$ and $d_j+G^X=\bar d_j$, pairwise orthogonal on each side. Let $\alpha:\C^N\to A$, $\beta:\C^N\to B$, $\bar\alpha:\C^N\to A/G$ and $\bar\beta:\C^N\to B/G^X$ be the linear maps defined by
		\begin{align*}
			\alpha(e_j):=c_j,  & \qquad  \beta(e_j):=d_j,\\
			\bar\alpha(e_j):=\bar c_j,  & \qquad  \bar\beta(e_j):=\bar d_j.
		\end{align*}
		Since the images of the coordinate vectors under each of the maps $\alpha$, $\beta$, $\bar\alpha$ and $\bar\beta$ are pairwise orthogonal positive norm-one elements, it suffices to show that $\alpha$ is completely isometric. Fix $s\in\N$ and $U_1,\ldots,U_N\in M_s(\C)$. Since $c_ic_j=0$ for distinct $i,j\in[N]$, the positive elements $U_j^*U_j\otimes c_j^2\in M_s(A)$, $1\le j\le N$, are pairwise orthogonal, and hence
		\begin{align*}
			\left\|\alpha^{[s]}\left(\sum_{j=1}^N U_j\otimes e_j\right)\right\|^2
			&=\left\|\sum_{j=1}^N U_j\otimes c_j\right\|^2
			=\left\|\sum_{j=1}^N U_j^*U_j\otimes c_j^2\right\|
			=\max_{1\le j\le N}\|U_j\|^2\\
			&=\left\|\sum_{j=1}^N U_j\otimes e_j\right\|^2.
		\end{align*}
		Thus $\alpha^{[s]}$ is isometric. Since $s\in\N$ was arbitrary, $\alpha$
		is completely isometric. The same calculation with $d_j$, $\bar c_j$ and
		$\bar d_j$ in place of $c_j$ shows that $\beta$, $\bar\alpha$ and
		$\bar\beta$ are completely isometric. Composing $\alpha\otimes_h\beta$ with
		the quotient map $A\otimes_hB\to A\otimes_{Z_X,h}B$ gives a complete
		contraction $\Lambda$; explicitly,
		\[
		\Lambda(e_i\otimes e_j):=c_i\otimes_{Z_X}d_j, \qquad i,j\in[N].
		\]
		For $i,j\in[N]$,
		\[
		q_G(\Lambda(e_i\otimes e_j))\stackrel{\eqref{eq:balanced-quotient-X}}{=}
		(c_i+G)\otimes(d_j+G^X)=\bar c_i\otimes\bar d_j.
		\]
		Hence, by linearity and continuity,
		\begin{equation}\label{eq:coordinate-quotient}
			q_G\circ\Lambda=\bar\alpha\otimes_h\bar\beta.
		\end{equation}
		Thus, for $s\in\N$ and $W\in M_s(\C^N\otimes_h\C^N)$,
		\[
		\|W\| \stackrel{\eqref{eq:haagerup-injectivity}}{=}
		\|(\bar\alpha\otimes_h\bar\beta)^{[s]}(W)\|
		\stackrel{\eqref{eq:coordinate-quotient}}{=}
		\|(q_G\circ\Lambda)^{[s]}(W)\|\le\|\Lambda^{[s]}(W)\|\le\|W\|.
		\]
		Hence $\Lambda$ is completely isometric. Its range is contained in $A\odot_{Z_X}B$ because its finite-dimensional domain equals $\C^N\odot\C^N$ as a vector space.
		
		Since $\Lambda(x\otimes y)=\alpha(x)\otimes_{Z_X}\beta(y)$ for $x,y\in\C^N$, every representation of $w\in\C^N\odot\C^N$ by elementary tensors gives a representation of $\Lambda(w)$ with the same number of terms. Hence $\len(\Lambda(w))\le\len(w)$.
		
		Fix $w\in\C^N\odot\C^N$. For $P\in\Prim(A)$, set 
		\[
		S(P):=\{j\in[N]:I_j\not\subseteq P\}.
		\]
		Since primitive ideals are prime and $I_1\cdots I_N=0\subseteq P$, repeated application of primeness gives $I_j\subseteq P$ for some $j\in[N]$, and hence $S(P)\subsetneq[N]$. If $j\in[N]\setminus S(P)$, the Rieffel correspondence gives $I_j^X\subseteq P^X$, so both $c_j+P$ and $d_j+P^X$ vanish. Define
		\begin{align*}
			\alpha_P:\C^{S(P)}\to A/P,  &\qquad \alpha_P(e_j):=c_j+P,\\
			\beta_P:\C^{S(P)}\to B/P^X, & \qquad  \beta_P(e_j):=d_j+P^X,
		\end{align*}
		for $j\in S(P)$, where $e_j$ also denotes the corresponding coordinate vector in $\C^{S(P)}$. When $S(P)=\emptyset$, both maps are zero. Each is a complete contraction, since it is obtained by extending coordinates by zero into $\C^N$, applying $\alpha$ or $\beta$, and then passing to $A/P$ or $B/P^X$, respectively. If $\pi_P:A\to A/P$ and $\pi_{P^X}:B\to B/P^X$ are the quotient maps, then
		\[
		\pi_P\circ\alpha=\alpha_P\circ\rho_{S(P)}
		\qquad\text{and}\qquad
		\pi_{P^X}\circ\beta=\beta_P\circ\rho_{S(P)}.
		\]
		Combining these identities with the definition of $\Lambda$, we obtain
		\[
		q_P\circ\Lambda \stackrel{\eqref{eq:balanced-quotient-X}}{=}  
		(\alpha_P\otimes_h\beta_P)\circ(\rho_{S(P)}\otimes_h\rho_{S(P)}).
		\]
		Evaluating at $w$ gives
		\[
		(\Lambda(w))_P
		=(\alpha_P\otimes_h\beta_P)(w_{S(P)}).
		\]
		Therefore
		\[
		\|(\Lambda(w))_P\|_h\le\|w_{S(P)}\|_h
		\le\max_{S\subsetneq[N]}\|w_S\|_h.
		\]
		Taking the supremum in \eqref{eq:offdiag-quotients} proves \eqref{eq:coordinate-transfer}.
	\end{proof}
	
	\subsection{Finite-dimensional obstruction and proof of the main theorem}
	
	The remaining ingredient for the necessary implications is a finite-dimensional tensor whose Haagerup norm strictly decreases under every proper coordinate restriction. More precisely, if $N,m\in\N$ and $2\le N\le m^2+1$, there exists $w\in\C^N\odot\C^N$ with
	\begin{equation}\label{eq:finite-obstruction}
		\len(w)\le m,\qquad \|w\|_h=1,\qquad
		\max_{S\subsetneq[N]}\|w_S\|_h<1.
	\end{equation}
	For $N=2$, take $w:=e_1\otimes e_2\in\C^2\odot\C^2$. Then $\len(w)=1\le m$, $\|w\|_h=1$, and $w_S=0$ for every $S\subsetneq[2]$. Now suppose that $N\geq3$, and let $k\in\N$ be least such that $N-1\le k^2$. Then $2\le k\le m$ and
	\[
	(k-1)^2+2\le N\le k^2+1.
	\]
	Hence \cite[Lemma~16]{ArchboldSomersetTimoney2005} gives $w\in\C^N\odot\C^N$ with $\len(w)\le k\le m$, $\|w\|_h=1$, and $\|w_S\|_h<1$ for every nonempty proper subset $S\subsetneq[N]$. Since $w_\emptyset=0$, \eqref{eq:finite-obstruction} follows.
	
	\begin{proof}[Proof of Theorem~\ref{thm:main}]
		Put $D:=\Link(X)$ and $Z:=Z_X$, and let $p,q\in\M(D)$ be the full linking projections. Let $j:A\otimes_{Z,h}B\to D\otimes_{Z,h}D$ be the embedding from Lemma~\ref{lem:corner-calculus}, applied with $C=D$, $e=p$ and $f=q$. By Lemma~\ref{lem:morita-ideals}, each condition on Glimm ideals in the theorem is equivalent for $A$, $B$ and $D$.
		
		\smallskip
		
		\emph{Sufficiency.} Suppose that every Glimm ideal of $A$ is $2$-primal. Then every Glimm ideal of $D$ is $2$-primal, so $\theta_D$ is injective by Theorem~\ref{thm:diagonal}\textup{(i)}. If $u\in A\otimes_{Z,h}B$ and $\Theta_X(u)=0$, then
		\[
		\|\theta_D(j(u))\|_{\cb} \stackrel{\eqref{eq:corner-action-norm}}{=} \|\Theta_X(u)\|_{\cb}=0.
		\]
		Hence $j(u)=0$, and the isometry of $j$ gives $u=0$. This proves sufficiency in \textup{(i)}.
		
		For \textup{(ii)}, fix $\ell\in\N$ and suppose that every Glimm ideal of $A$ is $(\ell^2+1)$-primal. The same holds for $D$. If $u\in A\odot_Z B$ and $\len(u)\le\ell$, Lemma~\ref{lem:corner-calculus} gives $\len(j(u))=\len(u)\le\ell$. Therefore Theorem~\ref{thm:diagonal}\textup{(ii)} and \eqref{eq:corner-action-norm} give
		\[
		\|u\|_{Z,h}=\|j(u)\|_{Z,h}  \stackrel{\text{Theorem~\ref{thm:diagonal}\textup{(ii)}}}{=}
		\|\theta_D(j(u))\|_{\cb}  \stackrel{\eqref{eq:corner-action-norm}}{=} \|\Theta_X(u)\|_{\cb}.
		\]
		
		For \textup{(iii)}, if every Glimm ideal of $A$ is primal, the same is true of $D$. Hence $\theta_D$ is isometric by Theorem~\ref{thm:diagonal}\textup{(iii)} and therefore completely isometric by Proposition~\ref{prop:complete-isometry}. For $n\in\N$ and $U\in M_n(A\otimes_{Z,h}B)$, the complete isometry of $j$ and \eqref{eq:corner-action-norm} give
		\[
		\|\Theta_X^{[n]}(U)\|  \stackrel{\eqref{eq:corner-action-norm}}{=}
		\|\theta_D^{[n]}(j^{[n]}(U))\|=\|j^{[n]}(U)\|=\|U\|.
		\]
		Thus $\Theta_X$ is completely isometric, and hence isometric.
		
		\smallskip
		
		\emph{Necessity.} If $G\in\Glimm(A)$ is not $2$-primal, take the map $\Lambda$ from Proposition~\ref{prop:transport-obstruction} with $N=2$ and put $u:=\Lambda(e_1\otimes e_2)\in A\odot_Z B$. As noted after \eqref{eq:finite-obstruction}, $\|e_1\otimes e_2\|_h=1$ and $(e_1\otimes e_2)_S=0$ for every $S\subsetneq[2]$. Since $\Lambda$ is isometric, $\|u\|_{Z,h}=1$, while \eqref{eq:coordinate-transfer} gives $\|\Theta_X(u)\|_{\cb}=0$. Hence $u\ne 0$ and $\Theta_X(u)=0$, so $\Theta_X$ is not injective. This proves necessity in \textup{(i)}.
		
		For necessity in \textup{(ii)}, fix $\ell\in\N$ and suppose that $G\in\Glimm(A)$ is not $(\ell^2+1)$-primal. Let $N\in\N$, $N\geq2$, be least such that $G$ is not $N$-primal. Then $N\le\ell^2+1$, and if $N\geq3$, minimality gives $(N-1)$-primality. Take $\Lambda$ as in Proposition~\ref{prop:transport-obstruction} and choose $w\in\C^N\odot\C^N$ as in \eqref{eq:finite-obstruction} with $m:=\ell$. Put $u:=\Lambda(w)\in A\odot_Z B$. By
		\eqref{eq:coordinate-transfer}, \eqref{eq:finite-obstruction} and the
		isometry of $\Lambda$,
		\[
		\len(u)
		\stackrel{\eqref{eq:coordinate-transfer}}{\le} \len(w)
		\stackrel{\eqref{eq:finite-obstruction}}{\le}\ell,
		\qquad
		\|u\|_{Z,h}=\|w\|_h
		\stackrel{\eqref{eq:finite-obstruction}}{=}1.
		\]
		Moreover,
		\[
		\|\Theta_X(u)\|_{\cb}
		\stackrel{\eqref{eq:coordinate-transfer}}{\le}
		\max_{S\subsetneq[N]}\|w_S\|_h
		\stackrel{\eqref{eq:finite-obstruction}}{<}1.
		\]
		Hence norm preservation at length at most $\ell$ forces $(\ell^2+1)$-primality.
		
		Finally, suppose that $\Theta_X$ is isometric. The necessary implication in \textup{(ii)}, already proved, shows that every Glimm ideal of $A$ is $(\ell^2+1)$-primal for every $\ell\in\N$. Given $n\in\N$, $n\geq2$, choose $\ell\in\N$ with $n\le\ell^2+1$ and use the monotonicity of $n$-primality from Section~\ref{subsec:cstar-facts}. Every Glimm ideal is then $n$-primal for every $n\geq2$, and hence primal. This completes \textup{(iii)}.
	\end{proof}
	
	\section{Stabilized inverse constants}\label{sec:constants}
	
	For a nonzero $A$--$B$ imprimitivity bimodule $X={}_AX_B$, we compare the inverse constants associated with the coefficient maps of $M_n(X)$, $n\in\N$. We show that, at every finite tensor length and in the completed case, their supremum over $n\in\N$ agrees with the corresponding stabilized constants for $A$ and $B$ and with the corresponding constant for $X\otimes\mathcal K_0$.
	
	\subsection{Definitions and full-corner invariance}
	
	We first introduce the inverse constants and their stabilizations and record the basic properties needed below. We then obtain finite factorizations associated with full projections and use them to prove that, for a $C^*$-algebra $C$ and full projections $e,f\in\M(C)$, the stabilized inverse constants of $eCf$ coincide with those of $C$.
	
	Let $X={}_AX_B$ be a nonzero imprimitivity bimodule. For $r\in\N$, define the finite-length constant $L_r(X)$, the completed constant $L_\infty(X)$, and their stabilizations by
	\begin{equation}\label{eq:inverse-constants}
		\begin{aligned}
			L_r(X)&:=\sup
			\left\{\frac{\|u\|_{Z_X,h}}{\|\Theta_X(u)\|_{\cb}}: 0\ne u\in A\odot_{Z_X}B, \, \len(u)\le r  \right\},\\
			L_\infty(X)&:=\sup \left\{ 
			\frac{\|u\|_{Z_X,h}}{\|\Theta_X(u)\|_{\cb}}: 0\ne u\in A\otimes_{Z_X,h}B \right\},\\
			L_t^{\mathrm{st}}(X)&:=\sup_{n\in\N}L_t(M_n(X)),\qquad t\in\N\cup\{\infty\}.
		\end{aligned}
	\end{equation}
	A ratio with zero denominator is interpreted as $+\infty$.
	
	Accordingly, for a nonzero $C^*$-algebra $C$, regarded as the multiplication $C$--$C$ imprimitivity bimodule, $L_\infty(C)$ and $L_\infty^{\mathrm{st}}(C)$ are, respectively, the constants $L(C)$ and $L'(C)$ of \cite[Definition~4.1 and Remark~4.2]{ArchboldSomersetTimoney2009}. Realize $\mathcal K_0$ as $\mathcal K(H_0)$ for a fixed separable infinite-dimensional Hilbert space $H_0$. Archbold, Somerset and Timoney proved that $L_\infty^{\mathrm{st}}(C)$ depends only on the homeomorphism class of $\Prim(C)$ and that
	\[
	L_\infty^{\mathrm{st}}(C)=L_\infty(C\otimes\mathcal K_0);
	\]
	see \cite[Theorem~6.2 and Corollary~6.3]{ArchboldSomersetTimoney2009}. Thus the homeomorphism $\Prim(A)\cong\Prim(B)$ induced by the Rieffel correspondence and \cite[Theorem~6.2]{ArchboldSomersetTimoney2009} already give $L_\infty^{\mathrm{st}}(A)=L_\infty^{\mathrm{st}}(B)$. We do not use a finite-length analogue of Theorem~6.2: the comparison of $X$ with $A$ and $B$, and all finite-length stabilization identities, are proved directly in Theorem~\ref{thm:full-corner-constants} and Corollary~\ref{cor:constants}.
	
	The defining sets are nonempty. Indeed, if $P\in\Prim(A)$, $a\in A\setminus P$ and $b\in B\setminus P^X$, then 
	\[
	q_P(a\otimes_{Z_X}b)\stackrel{\eqref{eq:balanced-quotient-X}}{=}
	(a+P)\otimes(b+P^X)\ne 0.
	\]
	Since $\Theta_X$ is completely contractive, it follows that
	$L_r(X),L_\infty(X)\in[1,\infty]$. A finite value is exactly the optimal constant in the corresponding reverse norm estimate. Moreover,
	\begin{equation}\label{eq:density-constants}
		L_\infty(X)=\sup_{r\in\N}L_r(X) \qquad \text{and} \qquad
		L_\infty^{\mathrm{st}}(X)=\sup_{r\in\N}L_r^{\mathrm{st}}(X).
	\end{equation}
	Indeed, $L_r(X)\le L_\infty(X)$ for every $r\in\N$. Conversely, suppose that
	$\lambda:=\sup_{r\in\N}L_r(X)<\infty$. Since every
	$u\in A\odot_{Z_X}B$ has finite length,
	\[
	\|u\|_{Z_X,h}\le\lambda\|\Theta_X(u)\|_{\cb},
	\qquad u\in A\odot_{Z_X}B.
	\]
	By density and continuity, the same estimate holds for all
	$u\in A\otimes_{Z_X,h}B$, and hence $L_\infty(X)\le\lambda$. If
	$\sup_{r\in\N}L_r(X)=\infty$, the first equality follows immediately from
	$L_r(X)\le L_\infty(X)$ for every $r\in\N$. Applying the first equality to
	$M_n(X)$ for every $n\in\N$ and interchanging the suprema over $n$ and $r$
	proves the second.
	
	The following finite-factorization lemma is based on Brown's theorem~\cite[Theorem~2.1]{Brown1977}.
	
	\begin{lemma}\label{lem:finite-factorization}
		Let $C$ be a $C^*$-algebra and $e\in\M(C)$ a full projection. For a finite set $\mathcal F\subseteq C$ and $\varepsilon>0$, there are $d\in\N$ and a contraction $R\in R_d(Ce)$ such that
		\begin{align*}
			\phi_R:C\to M_d(eCe),\qquad & \phi_R(c):=R^*cR, \\
			\psi_R:M_d(eCe)\to C,\qquad & \psi_R(T):=RTR^*
		\end{align*}
		are complete contractions respecting the left and right $Z_C$-actions and satisfy
		\[
		\|c-(\psi_R\circ\phi_R)(c)\|<\varepsilon,
		\qquad c\in\mathcal F.
		\]
		Here the $Z_C$-actions on $M_d(eCe)$ are induced by $z\mapsto(eze)I_d$, $z\in Z_C$.
	\end{lemma}
	
	\begin{proof}
		If $C=0$ or $\mathcal F=\emptyset$, take $d:=1$ and $R:=0$. Otherwise, fullness of $e$ means that the right ideal $eC$ generates $C$ as a closed two-sided ideal. By \cite[Theorem~2.1 and its proof]{Brown1977}, applied to $eC$, there exist a directed set $\Lambda$ and an increasing approximate identity $(g_\lambda)_{\lambda\in\Lambda}$ of positive contractions in $C$, each of which is a finite sum of elements $x^*x$ with $x\in eC$. Choose $\lambda\in\Lambda$ such that 
		\[
		\|c-g_\lambda c\|< \frac{\varepsilon}{2} \qquad \text{and} \qquad \|c-cg_\lambda\|<\frac{\varepsilon}{2}, \qquad  c\in\mathcal F,
		\]
		and write
		\[
		g:=g_\lambda=\sum_{j=1}^d x_j^*x_j, \qquad d\in\N, \quad x_1, \ldots , x_d \in eC.
		\]
		Put $R:=[x_1^*\ \cdots\ x_d^*]\in R_d(Ce)$. Then $RR^*=g$, so $\|R\|^2=\|RR^*\|=\|g\|\le1$. Since $(\psi_R\circ\phi_R)(c)=gcg$ for $c\in C$,
		\[
		\|c-(\psi_R\circ\phi_R)(c)\|
		\le\|c-gc\|+\|g\|\,\|c-cg\|<\varepsilon,
		\qquad c\in\mathcal F.
		\]
		The maps $\phi_R$ and $\psi_R$ are completely contractive by multiplication with $R$ and $R^*$.
		
		If $z\in Z_C$, then, since $x_j\in eC$ for $1\le j\le d$, centrality of $z$ gives
		\[
		R(eze)I_d=zR \qquad\text{and}\qquad (eze)I_dR^*=R^*z.
		\]
		Hence, for $c\in C$ and $T\in M_d(eCe)$,
		\begin{align*}
			\phi_R(zc)&=(eze)I_d\phi_R(c), \qquad
			\phi_R(cz)=\phi_R(c)(eze)I_d,\\
			\psi_R((eze)I_dT)&=z\psi_R(T), \qquad
			\psi_R(T(eze)I_d)=\psi_R(T)z.
		\end{align*}
		Thus both maps respect the left and right $Z_C$-actions.
	\end{proof}
	
	\begin{theorem}\label{thm:full-corner-constants}
		Let $C$ be a nonzero $C^*$-algebra, let $e,f\in\M(C)$ be full projections, and let $Y:=eCf$, considered as an $eCe$--$fCf$ imprimitivity bimodule. Then, for every $r\in\N\cup\{\infty\}$,
		\[
		L_r^{\mathrm{st}}(Y)=L_r^{\mathrm{st}}(C).
		\]
	\end{theorem}
	
	\begin{proof}
		By \eqref{eq:density-constants}, it suffices to take $r\in\N$. Put $Z:=Z_C$, identifying the corner centres through \eqref{eq:corner-centre}. We first prove $L_r(C)\le L_r^{\mathrm{st}}(Y)$; it suffices to consider the case $\lambda:=L_r^{\mathrm{st}}(Y)<\infty$. Take
		$u=\sum_{j=1}^r a_j\otimes_Z b_j\in C\odot_Z C$, with $a_j,b_j\in C$ for $1\le j\le r$, allowing zero terms, and let $\varepsilon>0$. Set
		\begin{equation}\label{eq:eta}
			\eta:=\frac{\varepsilon}{1+\sum_{j=1}^r(\|a_j\|+\|b_j\|)}.
		\end{equation}
		Apply Lemma~\ref{lem:finite-factorization} to $\{a_1,\ldots,a_r\}$ with the full projection $e$, and to $\{b_1,\ldots,b_r\}$ with the full projection $f$, using $\eta$ in place of $\varepsilon$ in both applications. After adjoining zero entries to the shorter of the two rows, if necessary, we may assume that they have the same length $d\in\N$; thus $R\in R_{d}(Ce)$ and $S\in R_{d}(Cf)$.
		
		Define
		\[
		v:=\sum_{j=1}^r(R^*a_jR)\otimes_Z(S^*b_jS)
		\in M_d(eCe)\odot_Z M_d(fCf),
		\]
		where $z\in Z$ acts by $(eze)I_d$ and $(fzf)I_d$ in the two factors. The displayed representation has at most $r$ elementary terms in the coefficient algebras of $M_d(Y)$, so $\len(v)\le r$. By \eqref{eq:module-haagerup-functoriality},
		$\psi_R\otimes_{Z,h}\psi_S$ is completely contractive. Expanding each tensor difference and using the Haagerup cross-norm inequality and $\|(\psi_R\circ\phi_R)(a_j)\|\le\|a_j\|$, we obtain
		\begin{align*}
			\left\|u-(\psi_R\otimes_{Z,h}\psi_S)(v)\right\|_{Z,h}
			&\le\sum_{j=1}^r
			\|a_j-(\psi_R\circ\phi_R)(a_j)\|\,\|b_j\|+\sum_{j=1}^r \|a_j\|\,\|b_j-(\psi_S\circ\phi_S)(b_j)\|\\
			&\stackrel{\text{Lemma}~\ref{lem:finite-factorization}}{\le} \eta \, \sum_{j=1}^r(\|a_j\|+\|b_j\|) \stackrel{\eqref{eq:eta}}{<}\varepsilon.
		\end{align*}
		Consequently,
		\begin{equation}\label{eq:norm-recovery}
			\|u\|_{Z,h}\le\varepsilon+\|v\|_{Z,h}.
		\end{equation}
		For $W\in M_d(Y)$, multiplication gives
		\begin{equation}\label{eq:block-comparison}
			\Theta_{M_d(Y)}(v)(W)=\sum_{j=1}^r(R^*a_jR)W(S^*b_jS)=R^*\theta_C(u)(RWS^*)S.
		\end{equation}
		Both maps 
		\[
		M_d(Y)\to C, \qquad W\mapsto RWS^*, \qquad \text{and} \qquad C\to M_d(Y), \qquad c\mapsto R^*cS,
		\]
		are complete contractions. Hence \eqref{eq:block-comparison} gives $\|\Theta_{M_d(Y)}(v)\|_{\cb}\le\|\theta_C(u)\|_{\cb}$. Since $L_r(M_d(Y))\le\lambda<\infty$, \eqref{eq:inverse-constants} gives $\|v\|_{Z,h}\le\lambda\|\Theta_{M_d(Y)}(v)\|_{\cb}$ for $v\ne 0$, while the same inequality is immediate for $v=0$. Thus
		\[
		\|u\|_{Z,h}\stackrel{\eqref{eq:norm-recovery}}{\le}
		\varepsilon+\|v\|_{Z,h}\stackrel{\eqref{eq:inverse-constants}}{\le}\varepsilon+\lambda\|\Theta_{M_d(Y)}(v)\|_{\cb}
		\stackrel{\eqref{eq:block-comparison}}{\le}
		\varepsilon+\lambda\|\theta_C(u)\|_{\cb}.
		\]
		Since $u\in C\odot_Z C$ of length at most $r$ and $\varepsilon >0$ were arbitrary, it follows that $L_r(C)\le L_r^{\mathrm{st}}(Y)$.
		
		For $m\in\N$, apply the inequality just proved to $M_m(C)$ and its full projections $I_m\otimes e$ and $I_m\otimes f$. Their rectangular corner is $M_m(Y)$. For every $d\in\N$, the canonical matrix identification
		\cite[1.2.12]{BlecherLeMerdy2004}, obtained by grouping the indices, identifies $M_d(M_m(Y))$ completely isometrically with $M_{dm}(Y)$ and similarly identifies the corresponding coefficient algebras and $Z$-actions. Under these identifications, the central Haagerup norm and algebraic tensor length are preserved and the coefficient maps are intertwined, so
		\[
		L_r(M_d(M_m(Y)))=L_r(M_{dm}(Y)).
		\]
		Consequently,
		\begin{align*}
			L_r(M_m(C))&\le L_r^{\mathrm{st}}(M_m(Y))\stackrel{\eqref{eq:inverse-constants}}{=}
			\sup_{d\in\N}L_r(M_d(M_m(Y)))=\sup_{d\in\N}L_r(M_{dm}(Y))\\
			&\le L_r^{\mathrm{st}}(Y).
		\end{align*}
		Taking the supremum over $m\in\N$ gives $L_r^{\mathrm{st}}(C)\le L_r^{\mathrm{st}}(Y)$. Conversely, Lemma~\ref{lem:corner-calculus} in $M_m(C)$ preserves tensor norm, length and coefficient-map norm, so $L_r(M_m(Y))\le L_r(M_m(C))$ for every $m\in\N$. Taking the supremum over $m\in\N$ proves $L_r^{\mathrm{st}}(Y)\le L_r^{\mathrm{st}}(C)$.
	\end{proof}
	
	\subsection{Compact stabilization}
	
	We now pass from matrix amplifications to compact stabilization. Using the finite-rank matrix corners of $\mathcal K_0$, we prove that the inverse constants of $X\otimes\mathcal K_0$ coincide with the stabilized inverse constants of $X$ and hence with those of its two coefficient algebras.
	
	For a nonzero imprimitivity bimodule $X={}_AX_B$ with linking algebra $D:=\Link(X)$ and canonical linking projections $p,q\in\M(D)$, define
	\[
	X\otimes\mathcal K_0:=(p\otimes1)(D\otimes\mathcal K_0)(q\otimes1),
	\]
	where $1$ is the identity of $\M(\mathcal K_0)$. This is the \emph{external tensor product imprimitivity bimodule} with coefficient algebras $A\otimes\mathcal K_0$ and $B\otimes\mathcal K_0$; see \cite[Proposition~3.36]{RaeburnWilliams1998}. Its canonical operator-space structure is the one inherited from the displayed corner of $D\otimes\mathcal K_0$.
	
	\begin{corollary}\label{cor:constants}
		For a nonzero imprimitivity bimodule $X={}_AX_B$ and every $r\in\N\cup\{\infty\}$,
		\begin{equation}\label{eq:stabilized-constants}
			L_r(X\otimes\mathcal K_0)=L_r^{\mathrm{st}}(X)
			=L_r^{\mathrm{st}}(A)=L_r^{\mathrm{st}}(B).
		\end{equation}
	\end{corollary}
	
	\begin{proof}
		Put $D:=\Link(X)$, with full linking projections $p,q\in\M(D)$. Applying Theorem~\ref{thm:full-corner-constants} to $(p,q)$, $(p,p)$ and $(q,q)$ gives $L_r^{\mathrm{st}}(X)=L_r^{\mathrm{st}}(A)=L_r^{\mathrm{st}}(B)$. By \eqref{eq:density-constants}, it remains to prove the first equality in \eqref{eq:stabilized-constants} for $r\in\N$.
		
		Fix an orthonormal basis of $H_0$, and let $k_n\in\mathcal K_0$ be the projection onto its first $n$ vectors for $n\in\N$. Put $C:=D\otimes\mathcal K_0$ and $X_0:=X\otimes\mathcal K_0$. The projections $p\otimes k_n,q\otimes k_n\in\M(C)$ are full and their rectangular corner $(p\otimes k_n)C(q\otimes k_n)$ is canonically identified with $M_n(X)$. Indeed, fullness follows from  
		\[
		\cspan(DpD)=\cspan(DqD)=D \qquad \text{and} \qquad \cspan(\mathcal K_0k_n\mathcal K_0)=\mathcal K_0.
		\]
		The projections $p\otimes1$ and $q\otimes1$ are full for the same reason. By \eqref{eq:corner-centre}, compression identifies the multiplier centres of $M_n(A)$ and $M_n(B)$, $n\in\N$, and of $A\otimes\mathcal K_0$ and $B\otimes\mathcal K_0$ with $Z:=Z_C$.
		
		By \eqref{eq:module-haagerup-functoriality}, the coefficient inclusions induce a complete contraction
		\[
		J_n:M_n(A)\otimes_{Z,h}M_n(B)
		\to (A\otimes\mathcal K_0)\otimes_{Z,h}(B\otimes\mathcal K_0).
		\]
		On each elementary tensor $a\otimes_Z b$, with $a\in M_n(A)$ and $b\in M_n(B)$, the composition of $J_n$ with the corner embedding of its codomain into $C\otimes_{Z,h}C$ agrees with the corner embedding $M_n(A)\otimes_{Z,h}M_n(B)\to C\otimes_{Z,h}C$ from
		Lemma~\ref{lem:corner-calculus}. Since the linear span of such tensors is dense
		in $M_n(A)\otimes_{Z,h}M_n(B)$, the two maps coincide. Both corner embeddings are completely isometric and preserve algebraic length by that lemma, so the same is true of $J_n$. For $v\in M_n(A)\otimes_{Z,h}M_n(B)$, applying \eqref{eq:corner-action-norm} to the corners
		$(p\otimes k_n)C(q\otimes k_n)$ and $(p\otimes1)C(q\otimes1)$ gives
		\[
		\|\Theta_{M_n(X)}(v)\|_{\cb}=\|\Theta_{X_0}(J_n(v))\|_{\cb}.
		\] 
		Hence $L_r(M_n(X))\le L_r(X_0)$, giving $L_r^{\mathrm{st}}(X)\le L_r(X_0)$.
		
		For the reverse inequality, suppose $\lambda:=L_r^{\mathrm{st}}(X)<\infty$ and take
		\[
		u=\sum_{j=1}^r a_j\otimes_Z b_j\in (A\otimes\mathcal K_0)\odot_Z(B\otimes\mathcal K_0),
		\]
		with $a_j\in A\otimes\mathcal K_0$ and $b_j\in B\otimes\mathcal K_0$ for
		$1\le j\le r$. Allowing zero terms covers every tensor of length at most $r$. For $n\in\N$ and $1\le j\le r$, put
		\[
		a_j^{(n)}:=(p\otimes k_n)a_j(p\otimes k_n),\qquad
		b_j^{(n)}:=(q\otimes k_n)b_j(q\otimes k_n),\qquad
		v_n:=\sum_{j=1}^r a_j^{(n)}\otimes_Z b_j^{(n)}.
		\]
		Here $v_n$ is regarded as a tensor over the coefficient algebras of $M_n(X)$. For each $t\in\mathcal K_0$, $k_ntk_n\to t$ in norm. By density of $D\odot\mathcal K_0$ in $C$ and contractivity of the compressions, $(1_{\M(D)}\otimes k_n)c(1_{\M(D)}\otimes k_n)\to c$ in norm for every $c\in C$. Since $a_j$ and $b_j$ belong to the respective $p$- and $q$-corners, it follows that $a_j^{(n)}\to a_j$ and $b_j^{(n)}\to b_j$ for $1\le j\le r$. The cross-norm inequality, applied to the difference of each elementary term, gives $J_n(v_n)\to u$ in the central Haagerup norm. Complete contractivity also gives $\Theta_{X_0}(J_n(v_n))\to\Theta_{X_0}(u)$ in completely bounded norm. Since $\len(v_n)\le r$,
		\[
		\|J_n(v_n)\|_{Z,h}=\|v_n\|_{Z,h}
		\stackrel{\eqref{eq:inverse-constants}}{\le}\lambda\|\Theta_{M_n(X)}(v_n)\|_{\cb}
		\stackrel{\eqref{eq:corner-action-norm}}{=}\lambda\|\Theta_{X_0}(J_n(v_n))\|_{\cb}.
		\]
		Passing to the limit proves $\|u\|_{Z,h}\le\lambda\|\Theta_{X_0}(u)\|_{\cb}$ and hence $L_r(X_0)\le L_r^{\mathrm{st}}(X)$. If $L_r^{\mathrm{st}}(X)=\infty$, then
		$L_r^{\mathrm{st}}(X)\le L_r(X_0)$ forces $L_r(X_0)=\infty$. This completes the proof.
	\end{proof}
	
	For the completed constant, matrix stabilization cannot in general be omitted:
	\cite[Corollary~7.14]{ArchboldSomersetTimoney2009} gives a $C^*$-algebra $C$
	for which $L_\infty(C)\le2$ while $L_\infty^{\mathrm{st}}(C)=\infty$.
	
	We conclude by explaining how the injectivity, norm-preservation and stabilization results apply to TROs and Hilbert $C^*$-modules, and why the nonunital theory does not extend in general to multiplier coefficients.
	
	\begin{remark}\label{rem:scope}
		Let $H$ and $K$ be Hilbert spaces. A nonzero \emph{ternary ring of operators} (TRO) is a norm-closed subspace $V\subseteq\mathcal B(H,K)$ satisfying $VV^*V\subseteq V$. Then 
		\[
		A_V:=\cspan(VV^*) \qquad \text{and} \qquad B_V:=\cspan(V^*V)
		\]
		are $C^*$-algebras, and $V$ is naturally an $A_V$--$B_V$ imprimitivity bimodule with inner products ${}_{A_V}\langle x,y\rangle:=xy^*$ and $\langle x,y\rangle_{B_V}:=x^*y$ for $x,y\in V$. Its operator-space structure is the one inherited from $\mathcal B(H,K)$. Similarly, if $B$ is a $C^*$-algebra and $Y$ is a nonzero right Hilbert $B$-module, then $Y$ is an imprimitivity bimodule between $\mathcal K_B(Y)$ and $B_Y:=\cspan\langle Y,Y\rangle_B$; see \cite[8.1.2, 8.1.4 and 8.1.14]{BlecherLeMerdy2004}. Thus the results above apply directly to TROs and Hilbert $C^*$-modules with these coefficient algebras.
		
		For TROs and their linking $C^*$-algebras, preservation of primal and Glimm ideals under the ideal correspondence is established in \cite[Propositions~3.7 and~3.11]{RajpalKansal2026}; compare Lemma~\ref{lem:morita-ideals}. Embeddings and ideal structure of the ordinary Haagerup tensor product $V\otimes_h C$ of a TRO $V$ and a $C^*$-algebra $C$ have been studied in \cite{KansalKumar2023,KansalKumarIdeals2023,KansalKumar2026,RajpalKansal2026}. Our TRO application concerns a different object: $V$ is viewed as an imprimitivity bimodule between $A_V$ and $B_V$, with $Z_V$ denoting their common multiplier centre, and the relevant coefficient map is defined on the central Haagerup tensor product $A_V\otimes_{Z_V,h}B_V$.
		
		In the nonunital case, the distinction between the canonical coefficient algebras and their multiplier algebras is essential in general. Indeed, \cite[Example~3.10]{ArchboldSomersetTimoney2009} gives a nonunital $C^*$-algebra $C$ for which every Glimm ideal is primal, while the multiplier-coefficient map
		\[
		\M(C)\otimes_{Z_C,h}\M(C)\to\CB(C),\qquad
		a\otimes_{Z_C}b\mapsto(x\mapsto axb),
		\qquad a,b\in\M(C), \quad x\in C,
		\]
		is not injective. Hence the Glimm-ideal criteria in Theorem~\ref{thm:main} do not, in general, remain valid after passing from the canonical coefficient algebras to their multiplier algebras. For further results relating the primitive ideal space of $\M(C)$ to norms of elementary operators on a separable $C^*$-algebra $C$, see \cite{ArchboldSomerset2014}.
	\end{remark}
	
	\subsection*{Funding}
	This research was supported by the European Union -- NextGenerationEU through the National Recovery and Resilience Plan 2021--2026, Institutional Grants of the University of Zagreb Faculty of Science (IK IA 1.1.3. Impact4Math, PMF-CROFUND).
	
	\subsection*{AI-use disclosure}
	OpenAI GPT-5.6 Sol and GPT-6 Astra were used during the preparation of this manuscript for literature searches, proofreading, copyediting, and improvement of the mathematical exposition. A subsequent review with GPT-6 Astra also led to further improvements in several proofs: some intermediate steps were made more explicit, and in a few places the arguments were recast in a cleaner and more direct form, notably in the complete-isometry and stabilization parts of the paper. The author independently checked all resulting changes, together with the mathematical arguments and references, and takes full responsibility for the final content of the manuscript.

\end{document}